\documentclass[11pt]{article}
\usepackage{amsmath,latexsym,amsxtra,enumerate,
color, cancel}
\usepackage{bbm, dsfont}
\usepackage{enumitem}
\usepackage{amsthm}
\usepackage[english]{babel}
\usepackage[usenames,dvipsnames,svgnames,table]{xcolor}
\usepackage{times, amsfonts, amssymb,  amscd, graphicx,stmaryrd}

\RequirePackage[colorlinks,citecolor=blue,urlcolor=blue]{hyperref}
\usepackage[active]{srcltx}
\newcommand{\E}{\mathbb{E}}
\newcommand{\R}{\mathbb{R}}
\newcommand{\PP}{\mathbb{P}}
\newcommand{\equa}{\begin{eqnarray*}}
\newcommand{\tion}{\end{eqnarray*}}
\newcommand{\equal}{\begin{eqnarray}}
\newcommand{\tionl}{\end{eqnarray}}
\theoremstyle{plain}
\newtheorem{thm}{Theorem}[section]
\newtheorem{lemma}[thm]{Lemma}
\newtheorem{prop}[thm]{Proposition}

\theoremstyle{definition}
\newtheorem{defn}[thm]{Definition}

\newtheorem{rem}[thm]{Remark}

\definecolor{color12}{rgb}{0.4, 0.8,0.4}

\newcommand{\ch}{}

\title{Existence, Uniqueness and Comparison Results for BSDEs with L\'evy Jumps in an Extended Monotonic Generator Setting}

\author{ Christel Geiss$^1$   \hspace{1.0em}   Alexander Steinicke$^2$\\\\
\small  
}
\date{}
\begin{document}
\parindent 0pt

\maketitle
\begin{abstract}
We show that the comparison results for a backward SDE  with jumps established in Royer (2006) and  Yin and Mao (2008) hold under more simplified conditions.
Moreover, we prove existence and uniqueness 
allowing the coefficients in the linear growth- and  monotonicity-condition for the generator  to be random and time-dependent. In the $L^2$-case with linear growth, this also generalizes  the results 
of Kruse and Popier (2016). 
For the proof of the comparison result, we introduce an approximation technique: Given a BSDE  driven by Brownian motion and Poisson random measure, we approximate  it by BSDEs where the  Poisson random measure admits only jumps of  size larger than $1/n.$ 
\end{abstract}

\vspace{1em}
{\noindent \textit{Keywords:} Backward stochastic differential equation; L\'evy process; comparison theorem; existence and uniqueness\\
\noindent \textit{Mathematics Subject Classification:} 60H10
}
{\noindent
\footnotetext[1]{University of Jyvaskyla, Department of Mathematics and Statistics, P.O. Box 35,\\ \hspace*{1.5em}
 FI-40014 University of Jyvaskyla. \\ \hspace*{1.5em}
 {\tt christel.geiss{\rm@}jyu.fi}}
\footnotetext[2]{Department of Mathematics and Information Technology, Montanuniversitaet Leoben, Austria. \\ \hspace*{1.5em}  {\tt alexander.steinicke{\rm@}unileoben.ac.at}}}

\section{Introduction}
In this paper, we study backward stochastic differential equations (BSDEs) of the form

\begin{equation}\label{bsde0}
Y_t=\xi+\int_t^T f(s,Y_s,Z_s,U_s)ds-\int_t^T Z_s dW_s-\int_{{]t,T]}\times(\R\setminus\{0\})}U_s(x)\tilde{N}(ds,dx),
\end{equation}
where  $W$ denotes   a one-dimensional Brownian motion  and $\tilde{N}$ a compensated Poisson random measure
belonging to  a given  L\'evy process with L\'evy measure $\nu$.  In particular, our focus lies on comparison results and existence and uniqueness of solutions. 

\bigskip

Comparison theorems state that---under certain conditions---if $\xi\le \xi'$ and $f\le f'$, then the process $Y$ of the solution satisfies $Y_t\le Y_t'$ for all $t\in{[0,T]}$. 
These types of theorems in the case of one-dimensional, Brownian BSDEs 
has been treated by Peng \cite{peng}, El Karoui et al. \cite{elkaroui}, 
\cite{carmona} and Cao and Yan \cite{caoyan}. 

\bigskip
Barles et al. give  in  \cite[Remark 2.7]{bbp}  a counterexample which shows  that in the jump case the
conditions $\xi \le \xi'$ and $f \le f'$ are not sufficient to guarantee $Y \le Y'$. They propose an additional 
sufficient condition which has been generalized 
 by Kruse and Popier \cite{Kruse}  , Royer \cite{royer}, Yin and Mao \cite{YinMao}, Becherer et al.~\cite{Bech} (allowing more general jump processes), and  Cohen et al. \cite{cohen} (for BSDEs driven by martingales). 
The condition of Kruse and Popier \cite{Kruse}  reads (in our $L^2$-setting) as follows: for each $s,y,z,u, u' \in [0,T]\times \R  \times  \R \times L^2(\nu)\times  L^2(\nu)$
there is a progressively measurable process $\gamma^{y,z,u,u'}\colon\Omega\times{[0,T]}\times\R\setminus\{0\}\to\R$  such that
\equal \label{Agamr}
&& f(s,y,z,u)-f(s,y,z,u')\leq \int_{\R\setminus\{0\}}(u(x)-u'(x))\gamma_s^{y,z,u,u'}(x)\nu(dx), \notag \\
&& -1 \leq \gamma^{y,z,u,u'}_s(x) \quad \text{and} \quad \sup_{s,\omega, y,z,u,u'}  |\gamma^{y,z,u,u'}_s|  \in L^2(\nu). 
\tionl

\bigskip 
One of the main results in the present paper is Theorem \ref{comparison} which states that  \eqref{Agamr}  can be replaced by the  simpler  condition
{\ch \equal \label{Agam-new} 
&&f(s,y,z,u)- f(s,y,z,u')\leq \int_{\R\setminus\{0\}}(u'(x)-u(x))\nu(dx), \quad \mathbb{P}\otimes\lambda\text{-a.e.} \notag \\
&&\text{for all} \,\, u, u'\in L^2(\nu) \, \text{ with } \,  u\leq u'. 
\tionl }
Notice that the r.h.s.~is infinite for $u'(x)-u(x) \notin L^1(\nu).$ 
Clearly,  \eqref{Agam-new}  is a weaker condition than \eqref{Agamr}, because one only needs to check the inequality for those $u, u'\in L^2(\nu)$ for 
which $u\le u'$ holds. Moreover, we do not need any $L^2(\nu)$ condition for $\gamma^{y,z,u,u'}_s$ but we choose $\gamma^{y,z,u,u'}_s(x)=-1.$ {\ch Under the 
constraint $ -1 \leq \gamma^{y,z,u,u'}_s(x),$ the choice  $\gamma^{y,z,u,u'}_s(x)=-1$ yields for $u'-u \ge 0$ the largest possible expression on the r.h.s.~of  \eqref{Agamr}, so that  \eqref{Agam-new}  can be seen as  the weakest possible condition which \eqref{Agamr} could impose on $f.$} \\
For a finite  L\'evy measure $\nu,$   Theorem \ref{comparison}  can be shown using only  elementary means. \smallskip 

Another  main result is a  method of how to  approximate a BSDE driven by a  L\'evy process with an infinite measure $\nu,$   by a sequence of BSDEs where the driving  processes  have a   finite  L\'evy measure.  We apply this result  to show   the comparison theorem for  BSDEs  driven by a general   L\'evy process. 
The proof relies on the Jankov--von Neumann theorem on  measurable sections/uniformizations  (this theorem is also important for dynamic programming, see   
El Karoui and Tan \cite{alkaroui-jankov}).  
Under certain conditions on the generator, the approximating solutions can 
be interpreted as nonlinear conditional expectations (in the sense of Peng \cite{pengII}), conditioned on a L\'evy process whose jumps are not of arbitrarily small size. 
(See the comments after Theorem \ref{finlev}.)

Studying the   existence, uniqueness, and comparison results by   Darling and Pardoux \cite{darlingpardoux}, 
Pardoux and Zhang \cite{pardouxzhang}, Pardoux \cite{pardoux}, Fan and Jiang \cite{fanjiang}, Royer \cite{royer}, Situ \cite{situ}, Yin and Mao \cite{YinMao}, Kruse and Popier 
\cite{Kruse}, \cite{Kruse17},  Yao \cite{Yao}, and  Sow \cite{bamba}, one notices that one can  unify and generalize  the assumptions on $f.$  \\
Indeed,  and this is our third main result, in  the case of $L^2$-solutions, for a  progressively measurable generator $f$ with linear growth,
it suffices to assume  (cf. Theorem  \ref{existence}     and  \ref{comparison}) the following growth- and monotonicity conditions with time-dependent, random coefficients: 
\begin{itemize}
\item $|f(\omega,s,y,z,u)|\leq F(s, \omega)+K_1(s,\omega )|y|+K_2(s,\omega)(|z|+\|u\|)$,
\item $(y-y')(f_1(\omega,s,y,z,u)-f_1(\omega,s,y',z',u'))\\
\hspace*{6em}\leq \alpha(s)\rho(|y-y'|^2)+\beta(s,\omega)|y-y'|(|z-z'|+\|u-u'\|)$,
\end{itemize}
with $\alpha\in L^1([0,T])$ and $F$ being nonnegative and progressively measurable such that  $\E\left [ \left ( \int_0^T F(\omega, t)dt \right )^2\right ] < \infty.$
 The     processes $K_1, K_2,$ and $\beta$ are  nonnegative and progressively measurable such that for a constant $c>0$,
$$\int_0^T(K_1(s)+K_2(s)^2+\beta(s)^2)ds <c,\quad\mathbb{P}\text{-a.s.}$$
The concave function $\rho$  in the monotonicity condition may  grow faster than linear at zero and satisfies $\int_{0^+} 1/\rho(x)dx=\infty.$ This type of function already appeared in context with BSDEs in Mao \cite{mao} in 1997. 

These assumptions also extend  the monotonicity condition of \cite{Kruse}, \cite{Kruse17}, for the $L^2$-case with linear growth, since the coefficients in our setting take randomness, 
the function $\rho$ and time-dependence into account.  BSDEs  with time-dependent coefficients appear, for example, in Gobet and Turkedjiev \cite{GobetTurk}. \bigskip

The existence and uniqueness result Theorem \ref{existence} and the comparison result Theorem \ref{comparison} are basic tools in the forthcoming paper 
\cite{GeissSteinIII}  on Malliavin differentiability  and boundedness of solutions to BSDEs.  To compute the  Malliavin  derivative for the jump part
of the L\'evy process,  more structure from the generator is required in its dependency on $u,$ usually via an integral w.r.t.~$\nu(dx),$ for example, 
 $$
 f(s,u) =    h  \Big(s,  \int_{\R\setminus\{0\}} u(x) \kappa(s,x)  \nu(dx) \Big),
   $$
 where $[0,T] \times \R \ni (s,v) \mapsto  h(s,v).$ 
 One can find $h$ and $\kappa$ such that the assumptions of Theorem  \ref{comparison} are satisfied while  conditon \eqref{Agamr} does not hold: 
  By the mean value theorem there exists a $\zeta \in ]0,1[$  and
 $$v_\zeta:=  \int_{\R\setminus\{0\}}  (\zeta  u(x)  + (1-\zeta) u'(x) ) \kappa(s,x)  \nu(dx),$$
 such that 
 \equa  f(s, u) - f(s, u') &=& \partial_v  h ( s, v_\zeta)   
  \int_{\R\setminus\{0\}}  (u(x)  - u'(x) ) \kappa(s,x)  \nu(dx).
 \tion
Assumption \eqref{Agam-new} holds
if  $ \gamma_s^{u,u'}(x):= \partial_v  h (s,v_\zeta)  \kappa(s,x) \ge -1$ for all $(s,u,u',x).$ {\ch  Choosing, for example, 
a bounded  function $h$ such that also $ \sup_{s,v} |\partial_v  h(s,v)|<\infty,$  but  $\partial_v  h(s,v) \neq 0$  for  a.e.~$s$ and $v,$ and  putting   $\kappa(s,x) = s^{-\frac{1}{4}} (|x|\wedge 1),$ then 
\eqref{Agamr} does not hold since
$$\sup_{s, u,u'} |\gamma_s^{u,u'} | \notin  L^2(\nu).$$ 
However,  the  assumptions \ref{2}, \ref{3} of section \ref{sec3} are  satisfied for }
$$K_2( s) = \beta(s) = \sup_v | \partial_v h  (s,v)| \|\kappa(s,\cdot)\|_{L^2(\nu)} \le c s^{-\frac{1}{4}}.$$

\bigskip
The paper is structured as follows: Section \ref{sec2} contains preliminaries and basic definitions. In section \ref{sec3}, we present the 
main theorems of this paper about existence and uniqueness of solutions, the approximation using BSDEs  based on L\'evy processes with finite L\'evy measure,  and the comparison result. The latter we also prove there.
Having  stated and proved  some auxiliary results in section \ref{sec4}, including an a-priori estimate for our type of BSDEs, we are able to prove 
 existence and uniqueness and the  approximation result from section \ref{sec3}.
 In the appendix, we recall the Bihari--LaSalle inequality and the Jankov--von Neumann theorem. 
 
\section{Setting}\label{sec2}

Let $X=\left(X_t\right)_{t\in{[0,T]}}$ be a c\`adl\`ag L\'evy process on a complete probability space $(\Omega,\mathcal{F},\mathbb{P})$
with L\'evy measure $\nu$. We will denote the augmented natural filtration of $X$ by
$\left({\mathcal{F}_t}\right)_{t\in{[0,T]}}$ and assume that $\mathcal{F}=\mathcal{F}_T.$  For $0<p\leq \infty$ we use the notation $(L^p,\|\cdot\|_p):=\left(L^p(\Omega,\mathcal{F},\mathbb{P}),\|\cdot\|_{L^p}\right)$.
Equations or inequalities for objects of these spaces throughout the paper are considered up to $\mathbb{P}$-null sets. \smallskip \\
The L\'evy--It\^o decomposition of a L\'evy process $X$ can be written as
\begin{equation}\label{LevyIto}
X_t = a t + \sigma W_t   +  \int_{{]0,t]}\times \{ |x|\le1\}} x\tilde{N}(ds,dx) +  \int_{{]0,t]}\times \{ |x|> 1\}} x  N(ds,dx),
\end{equation}
where $a\in\R$, $\sigma\geq 0$, $W$ is a Brownian motion and $N$ ($\tilde N$) is the (compensated) Poisson random measure corresponding to $X$, see \cite{applebaum} or \cite{satou}.\bigskip

{\bf Notation }
\begin{itemize}
\item Let  $\mathcal{S}^2$ denote the  space of all $(\mathcal{F}_t)$-progressively measurable and c\`adl\`ag processes  $Y\colon\Omega\times{[0,T]} \rightarrow \R$ such that
\equa
\left\|Y\right\|^2_{\mathcal{S}^2}:=\E\sup_{0\leq t\leq T} \left|Y_{t}\right|^2  <\infty.
\tion

\item We define $L^2(W) $ as the space of all $(\mathcal{F}_t)$-progressively measurable processes $Z\colon \Omega\times{[0,T]}\rightarrow \R$  such that
\equa
\left\|Z\right\|_{L^2(W) }^2:=\E\int_0^T\left|Z_s\right|^2 ds<\infty.
\tion

\item Let $\R_0:= \R\!\setminus\!\{0\}$. We define $L^2(\tilde N)$ as the space of all random fields $U\colon \Omega\times{[0,T]}\times{\R_0}\rightarrow \R$ 
which are measurable with respect to
$\mathcal{P}\otimes\mathcal{B}(\R_0)$ (where $\mathcal{P}$ denotes the predictable $\sigma$-algebra on $\Omega\times[0,T]$ generated
by the left-continuous $(\mathcal{F}_t)$-adapted processes) such that
\equa
\left\|U\right\|_{L^2(\tilde N) }^2:=\E\int_{{[0,T]}\times{\R_0}}\left|U_s(x)\right|^2 ds\,\nu(dx)<\infty.
\tion

\item $L^2(\nu):= L^2(\R_0, \mathcal{B}(\R_0), \nu),$ $\|\cdot \|:=\|\cdot \|_{L^2(\nu)}.$

\item { $L^p([0,T]):=L^p([0,T],\mathcal{B}([0,T]), \lambda)$ for $p>0$, where $\lambda$ is the Lebesgue measure on ${[0,T]}$.}

\item With a slight abuse of the notation, we define
\equal \label{LtwoLone}
&& \hspace{-3em} L^2(\Omega; L^1([0,T])) \\ \hspace{-5em}& \hspace{-3em}:=&\hspace{-1em}\!\!\!\!\! \left  \{F \in L^0(\Omega \times [0, T], \mathcal{F} \otimes \mathcal{B}([0,T]), \mathbb{P} \otimes \lambda):  \E \!\left[\int_0^T \!\!|F(\omega, t)| dt\right]^2\!\! <\! \infty.  \right \}  \notag
\tionl
For $F \in  L^2(\Omega; L^1([0,T])),$ put
\equal \label{I-FandK-F} 
I_F(\omega):= \int_0^T F(\omega, t)dt \quad \text{ and } \quad K_F(\omega, s) := \frac{F(\omega, s)}{I_F(\omega)}.
\tionl

\item A {\bf solution to a BSDE} with terminal condition $\xi$ and generator $f$ is a triplet $(Y,Z,U)\in \mathcal{S}^2\times L^2(W)\times L^2(\tilde N)$ which satisfies for all $t\in{[0,T]}$:
\begin{equation}\label{bsde}
Y_t=\xi+\int_t^T f(s,Y_s,Z_s,U_s)ds-\int_t^T Z_s dW_s-\int_{{]t,T]}\times\R_0}U_s(x)\tilde{N}(ds,dx).
\end{equation}

The BSDE \eqref{bsde} itself will be denoted by $(\xi,f)$.
\end{itemize}
\section{Main Results}\label{sec3}

We start with a result about existence and uniqueness  which is proved in section \ref{sec5}.

\begin{thm}\label{existence}
There exists a unique solution to the BSDE $(\xi,f)$ with $\xi\in L^2$ and generator $f:\Omega\times{[0,T]}\times\mathbb{R}\times\mathbb{R}\times L^2(\nu)\to\mathbb{R}$
satisfying the properties
\begin{enumerate}[label={(A\,\arabic*)}]
\item\label{1} For all $(y,z,u): (\omega,s)\mapsto f(\omega,s,y,z,u)$ is progressively measurable.

\item\label{2} There are nonnegative, progressively measurable processes $K_1,  K_2,$ and $F$ with 
\equal \label{C-K}
C_K:= \left \|\int_0^T\left(K_1(\cdot ,s)+K_2(\cdot,s)^2\right)ds \right \|_{\infty}<\infty
\tionl
 and   $F\in L^2(\Omega; L^1([0,T]))$  (see \eqref{LtwoLone}) such that for all $(y,z,u)$,
\begin{align*}
&|f(s,y,z,u)|\leq F(s)+K_1(s)|y|+K_2(s)(|z|+\|u\|), \quad \mathbb{P}\otimes\lambda\text{-a.e.}
 \end{align*}
\item\label{3} For $\lambda$-almost all $s$, the mapping $(y,z,u)\mapsto f(s,y,z,u)$ is $\mathbb{P}$-a.s. continuous. Moreover, there is a nonnegative function $\alpha\in L^1([0,T])$, $c>0$ and a progressively measurable process $\beta$ with $\int_0^T \beta(\omega,s)^2 ds<c$, $\mathbb{P}$-a.s. such that for all $(y,z,u), (y',z',u')$,
\begin{align*}
&(y-y')(f(s,y,z,u)-f(s,y',z',u'))\\
&\leq \alpha(s)\rho(|y-y'|^2)+\beta(s)|y-y'|(|z-z'|+\|u-u'\|), \mathbb{P}\otimes\lambda\text{-a.e.},
\end{align*}
where $\rho$ is a nondecreasing, continuous and concave function from ${[0,\infty[}$ to itself, satisfying $\rho(0)=0,$ and  $\int_{0^+}\frac{1}{\rho(x)}dx=\infty.$ 
 \item\label{4}  {\ch The function $\rho$ in \ref{3} satisfies
 $\limsup_{x \downarrow 0} \frac{\rho(x^2)}{x} =0.$}
\end{enumerate}
If $f$ satisfies only \ref{1}--\ref{3}, then there exists at most one solution.
\end{thm}

\bigskip
For  $\rho(x) =x,$ we are in the case of  the  ordinary monotonicity condition.   Another  example for  a function $\rho$ is given by 
$$\rho(x)=1-\min\left(x,\tfrac{1}{e}\right)^{\min\left(x,\tfrac{1}{e}\right)},\quad x\geq 0.$$
\begin{rem} {\color{white}.} \label{remark}
\begin{enumerate}
\item      
Condition \ref{2} implies that  $f(s,y,z,u)$ is integrable for a.e.~$s \in [0,T]$ since, by Fubini's theorem,
\equal  \label{f-integrable}
\int_0^T&& \hspace{-2em} \E |f(s,y,z,u)| ds \notag \\
 &\le& \E \int_0^T [F(s) +  K_1(s)|y| +  K_2(s)(|z|+\|u\|))]ds  < \infty. 
\tionl
\item {\ch If  $\limsup_{x \downarrow 0} \frac{\rho(x^2)}{x} =0$ is satisfied one can derive Lipschitz continuity of $f(s,y,z,u)$ in $z$ and $u$ 
      from the  monotonicity condition in \ref{3}. We require  \ref{4} since we later want to apply  \cite[Theorem 2.1]{YinMao},
  where Lip\-schitz continuity in $u$ is used to show uniqueness of solutions. If only \ref{1}--\ref{3} are satisfied but not \ref{4}, and a Lipschitz condition in $z, u$ holds nevertheless, all of the article's theorems remain valid.
 One can show  that   \ref{4}  
  does not follow from the other conditions imposed on $\rho$ in \ref{3}: Assume a decreasing sequence $(x_n)_{n=0}^\infty$ with $x_0=1$ and   $\lim_{n \to \infty} x_n =0.$ Define 
  $$\rho(x):= \left \{ \begin{array}{ll}   \sqrt{x_n} &\text{if }\,\,\,  x=x_n, \, n=0,1,2,... \\
               \sqrt{x}  &\text{if }\,\,\,  x > 1 \text{ or } x=0.
               \end{array} \right .
$$ 
and let $\rho$ be continuous and piecewise linear on $]0,1].$ The so defined  $\rho$ is a concave function  with $\limsup_{x \downarrow 0} \frac{\rho(x)}{\sqrt{x}} =1.$
The  sequence $(x_n)_{n=0}^\infty$ can be constructed such that $\int_{0}^1\frac{1}{\rho(x)}dx=\infty.$ For example, choose $x_1$ such that 
 $\int_{x_1}^{1}\frac{1}{\rho(x)}dx\ge 1,$ and if $x_n$ has been chosen find $x_{n+1}$ such that 
 $$\int_{x_{n+1}}^{x_n}\frac{1}{\rho(x)}dx  = \frac{1}{2}\left(\log(x_{n})-\log(x_{n+1})\right)\left(\sqrt{x_{n}}+\sqrt{x_{n+1}}\right)  \ge 1.$$  }
  \end{enumerate}
\end{rem}

The next result  shows how a solution  to a BSDE can be approximated by a sequence of  solutions of BSDEs which are  driven by   L\'evy processes with a 
finite L\'evy measure. 
We do this by approximating the underlying L\'evy process defined through 
$$X_t=at+\sigma W_t+\int_{]0,t]\times \{|x|>1\}}x N(ds,dx)+\int_{]0,t]\times \{|x|\leq 1\}}x \tilde{N}(ds,dx)$$
for $n\geq 1$ by
$$X^n_t=at+\sigma W_t+\int_{]0,t]\times \{|x|>1\}}x N(ds,dx)+\int_{]0,t]\times \{1/n \leq |x|\leq 1\}}x \tilde{N}(ds,dx).$$
The process $X^n$ has a finite L\'evy measure $\nu_n$. Furthermore, note that the compensated Poisson random measure associated with $X^n$ can be expressed as 
$\tilde{N}^n=\chi_{\{1/n \leq |x|\}}\tilde{N}.$ Let   
\equal  \label{J-n}
\mathcal{J}^0&:=&\{\Omega, \emptyset\} \vee \mathcal{N},  \notag \\
\mathcal{J}^n&:=& \sigma(X^n)  \vee \mathcal{N}, \quad n\geq 1,
\tionl
where $\mathcal{N}$ stands for  the null sets of $\mathcal{F}.$  Note that   $(\mathcal{J}^n)_{n=0}^\infty$ forms a filtration. The notation 
$(\mathcal{J}^n)_{n=0}^\infty$ was chosen to indicate that this filtration  describes  the inclusion of smaller and smaller {\it jumps} of the L\'evy process. We will use
$$   \E_n \cdot := \E\left[\ \cdot\ \middle|\mathcal{J}^n\right] $$ 
  for the conditional 
expectation. 
   \smallskip

The intuitive idea now would be to  work with a BSDE driven  by $X^n$ where one uses the data  $( \E_n \xi,  \E_n f).$  
The problem  is that
the generator $f$ needs to be progressively, and also jointly measurable w.r.t. $(\omega, t,y,z,u),$  but it is not 
obvious whether the conditional expectation $ \E_n f$  preserves this property from $f$.
For BSDEs driven by a Brownian motion, this problem has been solved in \cite[Proposition 7.3]{Ylinen}, but this proposition 
does not apply to our situtation. Therefore, we next propose  a method for the construction of a unique progressively measurable and jointly measurable w.r.t. 
$(\omega, t,y,z,u)$ version of $ \E_n f.$  \smallskip

\begin{defn}[Definition of $f_n$] \label{optional-f}  Assume that $f$ satisfies  \ref{1}, \ref{2} and  that 
$\mathbbm{J}:= \left(\mathcal{J}^{[s]}\right)_{s\in[0,\infty[}$  is built using  \eqref{J-n}, where $[\cdot]$ denotes the floor function.  
Let $\phantom{I}^{\!\!\!o,\mathbbm{J}\!\!} f$  be the optional projection of the process 
\begin{align*}
{[0,\infty[}\times\Omega\times{[0,T]}\times\R^2\times L^2(\nu) \to&\ \ \R,\\
(s,\omega,t,y,z,u)\mapsto& \ f(\omega,t,y,z,u)
\end{align*} 
in the variables $(s,\omega)$ with respect to 
$\mathbbm{J},$ and with  parameters $(t,y,z,u).$
For each $n\ge 0$, assume that the filtration $\mathbbm{F}^n :=\left(\mathcal{F}_t^n\right)_{t\in{[0,T]}}$ is given by $\mathcal{F}_t^n:=\mathcal{F}_t\cap\mathcal{J}^n.$
Let $f_n$ be the optional projection of 
$$(\omega,t,y,z,u)\mapsto \phantom{I}^{\!\!\!o,\mathbbm{J}\!\!} f(n,\omega,t,y,z,u)$$
 with respect to  $\mathbbm{F}^n$ with parameters $(y,z,u)$. 

\end{defn}
\bigskip
The reason for using the filtration  $\left(\mathcal{J}^{[s]}\right)_{s\in[0,\infty[}$  instead of the  $ (\mathcal{J}^n)_{n=0}^\infty$  from  \eqref{J-n}  is  
that one can apply  known measurability  results w.r.t.~right continuous filtrations instead of proving measurability here directly. Indeed,
the optional projection  $\phantom{I}^{\!\!\!o,\mathbbm{J}\!\!}f$ defined above  is jointly measurable in $(s,\omega,t,y,z,u).$  For this we refer to  \cite{meyer}, 
where  optional and predictable projections of random processes depending on parameters were considered, and their uniqueness up to indistinguishability
was shown. 

It follows that for all $(t,y,z,u)$, $$\phantom{I}^{\!\!\!o,\mathbbm{J}\!\!}f(n,t,y,z,u)=\E_n f(t,y,z,u), \quad \mathbb{P}\text{-a.s.} $$  

Then, since $f$ is $\left(\mathcal{F}_t\right)_{t\in{[0,T]}}$-progressively measurable, for all $n\geq 0$, $t\in {[0,T]}$ and all $(y,z,u)$, it holds that
\begin{equation}\label{optionalcond}
f_n(t,y,z,u)=\E_n f(t,y,z,u), \quad\mathbb{P}\text{-a.s.}
\end{equation}
Hence, $f_n(t,y,z,u)$ is a jointly measurable version of $\E_n f(t,y,z,u)$ which  is $\left(\mathcal{F}_t^n\right)_{t\in{[0,T]}}$-optional, so especially it is 
progressively measurable. \bigskip

 We comment on the compatibility of the solutions $(Y^n,Z^n,U^n)$ from the  BSDE corresponding to  $(\E_n\xi,f_n),$ 
\equa
Y^n_t &=&\E_n \xi+\int_t^T  f_n(s,Y^n_s,Z^n_s,U^n_s)ds-\int_t^T Z^n_s dW_s \\
&&-\int_{{]t,T]}\times \R_0}U^n_s(x)\tilde{N}^n(ds,dx)
 \tion
 with the space $S^2\times L^2(W)\times L^2(\tilde N)$:
 
The triplet $(Y^n,Z^n,U^n)\in S^2\times L^2(W)\times L^2(\tilde N^n)$ can be canonically embedded in the space $S^2\times L^2(W)\times L^2(\tilde N)$, basically by extending $U^n_s(x)$ 
onto $\R_0$ by defining $U^n_s(x):=0$  for $|x|<\frac{1}{n}$. 
Moreover, recall that $\tilde{N}^n=\chi_{\{1/n \leq |x|\}}\tilde{N},$ so that $$\int_{{]t,T]}\times \R_0}\!\!U^n_s(x)\tilde{N}^n(ds,dx) =\int_{{]t,T]}\times \R_0}\!\!\!U^n_s(x)  \chi_{\{1/n \leq |x|\}}\tilde{N}(ds,dx).$$
Therefore, $\left(Y^n,Z^n,U^n\chi_{\R\setminus{]-1/n,1/n[}}\right)$ solves $(\E_n\xi,f_n)$ in $S^2\times L^2(W)\times L^2(\tilde N)$.

\begin{thm}\label{finlev}
Let $\xi\in L^2$ and let $f$ satisfy \ref{1}--\ref{3}. 
Assume that  the BSDE driven by $X^n$ with data  $(\E_n\xi,f_n)$ (where $f_n$ is 
given by Definition \ref{optional-f}) has a  unique solution denoted by $(Y^n,Z^n,U^n).$    
If  the solution $(Y,Z,U)$ to $(\xi,f)$  exists as well, then,
 $$(Y^n,Z^n,U^n)\to(Y,Z,U)$$ in $L^2(W)\times L^2(W)\times L^2(\tilde{N})$ on $(\Omega,\mathcal{F},\mathbb{P})$.
Moreover, if $f$ additionally satisfies \ref{4}, then the  mentioned solution triplets exist.

\end{thm}

The benefit of this approximation becomes clear 
in the  proof of the comparison theorem which we state next. There, we only need to prove the comparison result
assuming a finite L\'evy measure, since the  general case then follows by approximation. \smallskip \\
Another consequence of this approximation result concerns nonlinear expectations. 
(For a survey article on nonlinear expectations the reader is referred to Peng \cite{pengII}.) In the case of L\'evy processes, 
provided that $f(s,y,0,0)=0$ for all $s$ and $y,$
the process $Y_t$ has been described by Royer in \cite{royer}  as a 
conditional nonlinear expectation, denoted by  $\E^f_t \xi:=Y_t.$  
Hence, our theorem implies that     $$(\E^{f_n}_t \E_n\xi )_{t \in [0,T]} \to  (\E^{f}_t \xi )_{t \in [0,T]} \quad \text{ in }   L^2(W) .$$

\begin{thm}\label{comparison}
Let $f,f'$ be two generators satisfying the conditions \ref{1}--\ref{3} of Theorem \ref{existence} ($f$ and $f'$ may have different coefficients). We assume $\xi\leq \xi'$, $\mathbb{P}$-a.s. and for all $(y,z,u)$, $f(s,y,z,u)\leq f'(s,y,z,u)$, for $\mathbb{P}\otimes\lambda$-a.a. $(\omega,s)\in\Omega\times{[0,T]}$.  Moreover, assume that $f$ or $f'$
satisfy the condition (here formulated for $f$)
\begin{enumerate}[label={(A\,$\gamma$)}] {\ch
\item\label{gamma}  $f(s,y,z,u)- f(s,y,z,u')\leq \int_{\R_0}(u'(x)-u(x))\nu(dx), \quad \mathbb{P}\otimes\lambda$-a.e. \\
\text{for all} \,\, $u, u'\in L^2(\nu) \, \text{ with } \,  u\leq u'. $ }
\end{enumerate}

Let $(Y,Z,U)$ and $(Y',Z',U')$ be the solutions to $(\xi,f)$ and $(\xi',f')$, respectively.

Then, $Y_t\leq Y'_t$, $\mathbb{P}$-a.s.
\end{thm}

\begin{proof}
The basic idea for this proof was inspired by the one of Theorem 8.3 in \cite{carmona}.

\texttt{Step 1:} \\
In this step we assume that the L\'evy measure $\nu$ is finite.
We use Tanaka--Meyer's formula (cf. \cite[Theorem 70]{protter}) to see that for $\eta(s):=2\beta(s)^2+\nu(\R_0)$,
\begin{align*}
&e^{\int_0^t \eta(s)ds}(Y_t-Y'_t)^2_+=e^{\int_0^T \eta(s)ds}(\xi-\xi')^2_++M(t)\\
& +\int_t^T e^{\int_0^s \eta(\tau)d\tau}\chi_{\{Y_s-Y'_s\geq 0\}}\biggl[2(Y_s-Y'_s)_+\left(f(s,Y_s,Z_s,U_s)-f'(s,Y'_s,Z'_s, U'_s)\right)\\
 &-|Z_s-Z'_s|^2- \eta(s)|Y_s-Y'_s|^2\\
 &-\int_{\R_0}\biggl((Y_s-Y'_s+U_s(x)-U'_s(x))^2_+-(Y_s-Y'_s)^2_+\\
&\quad\quad\quad\quad\quad\quad\quad-2(U_s(x)-U'_s(x))(Y_s-Y'_s)_+\biggr)\nu(dx)\biggr]ds.
\end{align*}
Here, $M(t)$ is a stochastic integral term having zero expectation which follows from  $Y,Y'\in\mathcal{S}^2$ (this holds
according to Theorem \ref{existence}). Moreover, we used that on the set $\{\Delta Y_s\geq 0\}$ (where  $\Delta Y:=Y-Y'$) we have 
$(Y_s-Y'_s)_+=|Y_s-Y'_s|$. Taking means and denoting the differences by  $\Delta \xi := \xi-\xi', \,\Delta Z := Z-Z', \, \Delta U:= U-U'$ and  $\Delta f:= f-f'$ leads us to
\begin{align}\label{TanMey}
&\E  e^{\int_0^t \eta(s)ds}(\Delta Y_t)^2_+=\E  e^{\int_0^T \eta(s)ds}(\Delta \xi)^2_+ \nonumber\\
&+\E \bigg \{ \int_t^T e^{\int_0^s \eta(\tau)d\tau}\chi_{\{\Delta Y_s\geq 0\}}\biggl[2(\Delta Y_s)_+\left(f(s,Y_s,Z_s,U_s)-f'(s,Y'_s,Z'_s, U'_s)\right)  \nonumber \\
& -|\Delta Z_s|^2- \eta(s)|\Delta Y_s|^2\nonumber\\
&-\int_{\R_0}\left((\Delta Y_s+ \Delta U_s(x))^2_+-(\Delta Y_s)^2_+-2(\Delta U_s(x))(\Delta Y_s)_+\right)\nu(dx)\biggr]ds \bigg \},
\end{align} 
 We split up the set $\R_0$ into $$B(\omega,s)=B=\{\Delta U_s(x)\geq  -\Delta Y_s\}\text{ and } B^c.$$
 Taking into account that $\xi\leq \xi'$, we estimate
\begin{align}\label{TanMeyf}
&\E e^{\int_0^t \eta(s)ds}(\Delta Y_t)^2_+\nonumber\\
\leq&\E \bigg \{ \int_t^T e^{\int_0^s \eta(\tau)d\tau}\chi_{\{\Delta Y_s\geq 0\}}\biggl[2(\Delta Y_s)_+\left(f(s,Y_s,Z_s,U_s)-f'(s,Y'_s,Z'_s, U'_s)\right)\\
& -|\Delta Z_s|^2- \eta(s)|\Delta Y_s|^2\nonumber\\
&-\!\int_{B}\!|\Delta U_s(x)|^2\nu(dx)+\int_{B^c}\!\left((\Delta Y_s)^2_++2(\Delta U_s(x))(\Delta Y_s)_+\right)\nu(dx)\biggr]ds \bigg \} .\nonumber
\end{align} 
 We focus on the term $(\Delta Y_s)_+\left(f(s,Y_s,Z_s,U_s)-f'(s,Y'_s,Z'_s, U'_s)\right)$, and denoting $((Y,Z),(Y',Z'))$ by $(\Theta,\Theta'),$
 we derive from   $f\leq f'$ that
 \begin{align*}
& (\Delta Y_s)_+\left(f(s,\Theta_s,U_s)-f'(s,\Theta'_s, U'_s)\right)\\
& =(\Delta Y_s)_+\left(f(s,\Theta_s,U_s)-f(s,\Theta'_s, U'_s)+f(s,\Theta'_s,U'_s)-f'(s,\Theta'_s, U'_s)\right)\nonumber\\
&\leq(\Delta Y_s)_+\left(f(s,\Theta_s,U_s)-f(s,\Theta'_s, U'_s)\right).\nonumber
 \end{align*}
We continue with the observation that on $\{\omega: \Delta Y_s >0 \}$ we have
$$ B^c = \{ \Delta U_s(x)< - \Delta Y_s\} \subseteq \{  U'_s(x) > U_s(x)\},  $$ 
so that 
       $$  U'_s\chi_B+U_s\chi_{B^c} \le   U'_s\chi_B+U'_s\chi_{B^c}   \quad \text{ on } \quad    \{\omega: \Delta Y_s >0 \}.    $$   
Therefore, we split  $(\Delta Y_s)_+\left(f(s,\Theta_s,U_s)-f(s,\Theta'_s, U'_s)\right)$  into two  terms; one we estimate with \ref{3} and the first inequality of \eqref{algebra}, while for the other  we use \ref{gamma}: 
\begin{align*}
& \hspace{-2em} (\Delta Y_s)_+\left(f(s,\Theta_s,U_s)-f(s,\Theta'_s, U'_s)\right)\nonumber\\
=&(\Delta Y_s)_+\left(f(s,\Theta_s,U_s\chi_B+U_s\chi_{B^c})-f(s,\Theta'_s, U'_s\chi_B+U_s\chi_{B^c})\right)\nonumber\\
&+(\Delta Y_s)_+\left(f(s,\Theta'_s,U'_s\chi_B+U_s\chi_{B^c})-f(s,\Theta'_s, U'_s\chi_B+U'_s\chi_{B^c})\right)\\
\leq&\, \alpha(s)\rho((\Delta Y_s)_+^2)+\beta(s)^2(\Delta Y_s)_+^2+\frac{|\Delta Z_s|^2}{2}+\frac{\|\Delta U_s\chi_B\|^2}{2}\nonumber\\
&- \int_{ \R_0}(\Delta Y_s)_+ \Delta U_s(x)\chi_{B^c}\nu(dx).\nonumber
\end{align*} 
Thus, by the last two inequalities, \eqref{TanMeyf} evolves to 
\begin{align*}
&\E e^{\int_0^t \eta(s)ds}(\Delta Y_t)^2_+\\
&\leq \ \E \bigg \{\!\!\int_t^T e^{\int_0^s  \eta (\tau)d\tau}\chi_{\{\Delta Y_s\geq 0\}}\biggl[2\alpha(s){\rho}((\Delta Y_s)_+^2)+2\beta(s)^2(\Delta Y_s)_+^2+|\Delta Z_s|^2\\
& \quad+\|\Delta U_s\chi_B\|^2 -\int_{B^c}2(\Delta Y_s)_+ (\Delta U_s(x))\nu(dx)-|\Delta Z_s|^2-\eta(s)|\Delta Y_s|^2\nonumber\\
&\quad-\!\int_{B}\!|\Delta U_s(x)|^2\nu(dx)+\int_{B^c}\!\left((\Delta Y_s)^2_++2(\Delta Y_s)_+(\Delta U_s(x))\right)\nu(dx)\biggr]ds \bigg \}.\nonumber
\end{align*} 
Because of $\|\Delta U_s\chi_B\|^2=\int_{B}\!|\Delta U_s(x)|^2\nu(dx)$, we cancel out terms and get
\begin{align*}
&\E e^{\int_0^t \eta (s)ds}(\Delta Y_t)^2_+\\
&\leq \ \E \bigg \{\int_t^T e^{\int_0^s \eta (\tau)d\tau}\chi_{\{\Delta Y_s\geq 0\}}\biggl[2\alpha(s){\rho}((\Delta Y_s)_+^2)+2\beta(s)^2(\Delta Y_s)_+^2\\
&\quad\quad\quad\quad - \eta(s)|\Delta Y_s|^2 +\int_{B^c}\!(\Delta Y_s)^2_+\nu(dx)\biggr]ds \bigg \}.\nonumber
\end{align*} 
Bounding $\int_{B^c}\!(\Delta Y_s)^2_+\nu(dx)$ by $\nu(\R_0)(\Delta Y_s)^2_+$, leads us to
\begin{align*}
&\E e^{\int_0^t \eta (s)ds}(\Delta Y_t)^2_+\\
&\leq \ \E \bigg \{\int_t^T e^{\int_0^s \eta(\tau)d\tau}\chi_{\{\Delta Y_s\geq 0\}}\biggl[2\alpha(s){\rho}((\Delta Y_s)_+^2)\\
&\quad +(2\beta(s)^2+\nu(\R_0))(\Delta Y_s)_+^2 -\eta(s)|\Delta Y_s|^2\biggr]ds \bigg \}.\nonumber
\end{align*} 
It remains, also using  the definition of $\eta$,
\begin{align*}
&\E e^{\int_0^t \eta (s)ds}(\Delta Y_t)^2_+\leq\E\int_t^T e^{\int_0^s \eta (\tau)d\tau}2\alpha(s){\rho}((\Delta Y_s)_+^2)ds.\nonumber
\end{align*}
The term $e^{\int_0^T \eta(\tau)d\tau}$ is $\mathbb{P}$-a.s. bounded by a constant $C>0$. Thus, by the concavity of $\rho$, we arrive at
\begin{align*}
&\E(\Delta Y_t)^2_+\leq\E \big [e^{\int_0^t \eta (s)ds}(\Delta Y_t)^2_+\big ]\leq\int_t^T 2C\alpha(s){\rho}(\E(\Delta Y_s)_+^2)ds.\nonumber
\end{align*}
Then, the Bihari--LaSalle inequality (Proposition \ref{bihari-prop})---a generalization of Gronwall's inequality---shows that $\E(\Delta Y_t)^2_+=0$ for all $t\in{[0,T]}$, which is the desired result for $\nu(\R_0)<\infty$.

\bigskip
\texttt{Step 2:}

The goal of this step is to extend the result of the first step to general L\'evy measures. We adapt the notation of Theorem \ref{finlev} for $Y^n, {Y^n}', f_n,$ 
and $f_n'$.
Now, we claim  that for solutions $Y^n$ and ${Y^n}'$ of $(\E_n\xi,f_n)$ and $(\E_n\xi',f_n'),$   Step 1 granted 
that $Y^n\leq {Y^n}':$ Indeed, $f_n\leq f_n'$ holds by the monotonicity of $\E_n$, and also \ref{gamma} holds for $f_n$ if it did for $f.$ One notes that
the process $X^n$ which is related to $(\E_n\xi,f_n)$ and $(\E_n\xi',f_n')$ has a finite L\'evy  measure $\nu_n$ satisfying $\nu_n(|x|<\frac{1}{n} )=0,$ while in   \ref{gamma} 
we still have $\nu.$  However, the solution processes $U^n$ and ${U^n}'$ are zero for $|x|<\frac{1}{n}$  (see the comment before Theorem \ref{finlev}).

Hence, we need   \ref{gamma}  only for $u$ and $u'$ which are zero for $|x|<\frac{1}{n},$ and for those $u$ and $u'$ we may replace $\nu$ by $\nu_n$ and then apply  Step 1. 
Finally, the convergence of the sequences to the solutions $Y$ and $Y'$ of $(\xi,f)$ and $(\xi',f')$, respectively, in $L^2(W)$ shows $Y\leq Y',$ and 
our theorem is proven.
 \end{proof}

\section{Auxiliary Results} \label{sec4}
We will frequently use the following basic algebraic inequalities (special cases of Young's inequality) which hold for all $R>0$:
\begin{equation}\label{algebra}
ab\leq \frac{a^2}{2R}+\frac{Rb^2}{2}\quad\quad\text{and}\quad\quad ab\leq \frac{Ra}{2}+\frac{ab^2}{2R}.
\end{equation}
The following proposition states, roughly speaking, that for the BSDEs considered here it is sufficient to find solution processes of a BSDE in the (larger) space $L^2(W)\times L^2(W)\times L^2(\tilde{N})$.
\begin{prop}\label{supprop}
If $(Y,Z,U)\in L^2(W)\times L^2(W)\times L^2(\tilde{N})$ is a triplet of processes that satisfies the BSDE $(\xi,f)$ with $\xi\in L^2$ and 
\ref{1}, \ref{2}, then $(Y,Z,U)$ is a solution to \eqref{bsde}, i.e., $(Y,Z,U)\in \mathcal{S}^2\times L^2(W)\times L^2(\tilde{N})$. In particular, 
there exists a  constant $C_1>0$ 
such that
\begin{align*}
&\|Y\|^2_{\mathcal{S}^2} +\left\|Z\right\|_{L^2(W) }^2 + \left\|U\right\|_{L^2(\tilde N) }^2  \leq e^{C_1(1+ C_K)^2} \left(\E |\xi|^2+\E I^2_F\right),
\end{align*}
where $C_K$ was defined in \eqref{C-K} and $I_F$ in \eqref{I-FandK-F}.
\end{prop}
\begin{proof}
Since $(Y,Z,U)$ satisfies \eqref{bsde}, it holds that
\begin{align*}
|Y_t|^2= &Y_t\xi+Y_t\int_t^T f(s,Y_s,Z_s,U_s)ds-Y_t\int_t^T Z_s dW_s\\
&-Y_t\int_{{]t,T]}\times\R_0}U_s(x)\tilde{N}(ds,dx).
\end{align*}
We apply the first inequality of \eqref{algebra}, where $Y_t$ takes the role of $a$, to get for an arbitrary $R>0$:
\begin{align*}
|Y_t|^2\leq&\frac{3|Y_t|^2}{2R}+\frac{R|\xi|^2}{2}+\frac{R}{2}\left(\left|\int_t^T Z_s dW_s\right|^2+\left|\int_{{]t,T]}\times\R_0}U_s(x)\tilde{N}(ds,dx)\right|^2\right)\\
&+|Y_t|\int_t^T |f(s,Y_s,Z_s,U_s)|ds.
\end{align*}
Condition \ref{2} implies
\begin{align*}
|Y_t|^2\leq&\frac{3|Y_t|^2}{2R}+\frac{R|\xi|^2}{2}+\frac{R}{2}\left(\left|\int_t^T Z_s dW_s\right|^2+\left|\int_{{]t,T]}\times\R_0}U_s(x)\tilde{N}(ds,dx)\right|^2\right)\\
&+|Y_t|\int_t^T \left(F(s)+K_1(s)|Y_s|+K_2(s)\left(|Z_s|+\|U_s\|\right)\right)ds.
\end{align*}
We estimate with the help of the inequalities \eqref{algebra},
\begin{align*}
|Y_t|F(s) &\leq K_F(s)\left(\frac{|Y_t|^2}{2R}+\frac{RI_F^2}{2}\right),\\
K_1(s)|Y_t||Y_s|&\leq K_1(s)\left(\frac{|Y_t|^2}{2R}+\frac{R|Y_s|^2}{2}\right),\\
|Y_t|K_2(s)\left(|Z_s|+\|U_s\|\right)&\leq \frac{K_2(s)^2|Y_t|^2}{2R}+R\left(|Z_s|^2+\|U_s\|^2\right).
\end{align*}
Hence,
\begin{align*}
|Y_t|^2\leq&\frac{|Y_t|^2}{2R}\left(4+\int_t^T\left(K_F(s)+K_1(s)+K_2(s)^2\right)ds\right)+\frac{R|\xi|^2}{2}\\
&+\frac{R}{2}\left(\left|\int_t^T Z_s dW_s\right|^2+\left|\int_{{]t,T]}\times\R_0}U_s(x)\tilde{N}(ds,dx)\right|^2\right)\\
&+\frac{R}{2}  I_F^2 \int_t^T K_F(s)ds+R\!\int_t^T\left(|Z_s|^2+\|U_s\|^2\right)ds+\!\int_t^T\!\!\frac{RK_1(s)|Y_s|^2}{2}ds.
\end{align*}
Note that $\int_0^T K_F(s)ds =1$ and choose $R=R_0:= 5+\int_0^T\left(K_1(s)+K_2(s)^2\right)ds$ so that
\begin{align*}
|Y_t|^2\leq&R_0\Biggl[|\xi|^2+\!\!\sup_{t\in{[0,T]}}\!\left(\left|\int_t^T Z_s dW_s\right|^2\!\!+\left|\int_{{]t,T]}\times\R_0}U_s(x)\tilde{N}(ds,dx)\right|^2\right)\\
&\quad \quad \quad \quad\quad \quad \quad  +I_F^2+2\int_0^T|Z_s|^2+\|U_s\|^2ds    +  \int_t^T K_1(s)|Y_s|^2ds \Biggr].
\end{align*}
Since $Y$ is a c\`adl\`ag process, we may apply \eqref{gronwall-backwards} from the appendix which
leads to
\begin{align*}
|Y_t|^2\leq&  R_0 e^{R_0 \int_0^T K_1(s)ds}      \Biggl[|\xi|^2 +I_F^2+2\int_0^T|Z_s|^2+\|U_s\|^2ds\\&+ \!\!\sup_{t\in{[0,T]}}\!\left(\left|\int_t^T Z_s dW_s\right|^2\!+\left|\int_{{]t,T]}\times\R_0}U_s(x)\tilde{N}(ds,dx)\right|^2\right) \Biggr].
\end{align*}
The inequality $(a+b)^2\leq 2a^2+2b^2$ and then Doob's martingale inequality used on 
\begin{align*}
\sup_{t\in{[0,T]}}\Bigg(&\bigg|\!\int_0^T Z_s dW_s-\!\int_0^t Z_s dW_s\bigg|^2\\
&+\bigg|\int_{{]0,T]}\times\R_0}U_s(x)\tilde{N}(ds,dx)-\int_{{]0,t]}\times\R_0}U_s(x)\tilde{N}(ds,dx)\bigg|^2\Bigg)
\end{align*}
yield, since a.s. $ R_0 \le 5 + C_K $ and   $\int_0^T K_1(s)ds \le C_K,$ 
\begin{align}\label{supreme}
\E\sup_{t\in{[0,T]}}|Y_t|^2
\leq\ c_1       \Biggl[\E|\xi|^2+  \E I_F^2+  12 \E\int_0^T\left(|Z_s|^2+\|U_s\|^2\right)ds\Biggr]
\end{align}
with 
\equal \label{the-constant}
c_1=(5 + C_K)e^{(5+C_K)C_K}.
\tionl
For a progressively measurable process $\eta$, which we will determine later, It\^o's formula implies that
\begin{align}\label{gammaeq1}
&{|Y_0|}^2+\int_0^T e^{\int_0^s\eta(\tau)d\tau}\left(\eta(s){|Y_s|}^2 +{|Z_s|}^2 +\|U_s\|^2\right)ds\nonumber\\
& =M(0)+e^{\int_0^T\eta(s)ds}|\xi|^2+\int_0^T 2e^{\int_0^s\eta(\tau)d\tau}Y_sf(s,Y_s,Z_s,U_s)ds,
\end{align}
where 
\begin{align}\label{mart}
M(t)=&-\int_t^T2e^{\int_0^s\eta(\tau)d\tau}Y_sZ_sdW_s\nonumber\\
&-\int_{{]t,T]}\times\R_0}2e^{\int_0^s\eta(\tau)d\tau}\left((Y_{s-}+ U_s(x))^2-Y_{s-}^2\right)\tilde{N}(ds,dx).
\end{align}
Provided  that $\left \| \int_0^T\eta(\tau)d\tau \right \|_{L^\infty(\mathbb{P})} < \infty,$ one gets  $\E M(t)=0$ as a consequence of  \eqref{supreme} and the 
Burkholder--Davis--Gundy inequality  (see, for instance, \cite[Theorem 10.36]{HeWangYan}), where the term $\left((Y_{s-}+ U_s(x))^2-Y_{s-}^2\right)^2$  appearing  
in the integrand can be estimated by
$$\left(|Y_{s-}+ U_s(x)|+|Y_{s-}|\right)^2 \left(|Y_{s-}+ U_s(x)|-|Y_{s-}|\right)^2  \le 4 \sup_{r \in [0,T]} |Y_r|^2  \,|U_s(x)|^2. $$ 
By \ref{2} and \eqref{algebra}, we have 
\equa
  |Y_s| |f(s,Y_s,Z_s,U_s)| &\le & |Y_s|[F(s)+K_1(s)|Y_s|+K_2(s)(|Z_s|+\|U_s\|) ]\\
   &\le &   F(s) |Y_s| + K_1(s)|Y_s|^2  +2R \frac{ K_2(s)^2 |Y_s|^2 }{2} \\
   &&+ \frac{|Z_s|^2+ \|U_s\|^2 }{2 R}.
\tion 
We use this estimate for $R=2,$  and taking the expectation in \eqref{gammaeq1}, we have 
\begin{align}\label{gammaeq2}
&\E\int_0^T e^{\int_0^s\eta(\tau)d\tau}\left(\eta(s){|Y_s|}^2 +{|Z_s|}^2 +\|U_s\|^2\right)ds\nonumber\\
& \leq \E e^{\int_0^T\eta(s)ds}|\xi|^2+\E\int_0^T e^{\int_0^s\eta(\tau)d\tau}\biggl(\frac{|Z_s|^2 +\|U_s\|^2}{2} \biggr)ds \\
&\quad+2\E\int_0^T e^{\int_0^s\eta(\tau)d\tau}F(s)ds\sup_{t\in{[0,T]}}|Y_t|\nonumber\\
&\quad +\E\int_0^T e^{\int_0^s\eta(\tau)d\tau} 2\left(K_1(s)+ 2K_2(s)^2\right)Y_s^2ds.\nonumber
\end{align}
Then, we choose $\eta(s)=2\left(K_1(s)+2K_2(s)^2\right)$ and subtract the terms containing $Y, Z,$ and $U$ from the left hand side of \eqref{gammaeq2}. 
Moreover, we apply the first inequality of \eqref{algebra} to the term containing the supremum. It follows that 
\begin{align}\label{gammaeq4}
&\E\int_0^T e^{\int_0^s\eta(\tau)d\tau}\left({|Z_s|}^2 +\|U_s\|^2\right)ds\nonumber\\
& \leq 2\E  \big [e^{\int_0^T\eta(s)ds}|\xi|^2\big ]+ 2R\E\left[\int_0^T e^{\int_0^s\eta(\tau)d\tau}F(s)ds\right]^2+\frac{2}{R}\E\sup_{t\in{[0,T]}}|Y_t|^2.
\end{align}

Note that 
$$\E\int_0^T \left({|Z_s|}^2 +\|U_s\|^2\right)ds\leq\E\int_0^T e^{\int_0^s\eta(\tau)d\tau}\left({|Z_s|}^2 +\|U_s\|^2\right)ds.$$
Hence, by \eqref{gammaeq4} and $\int_0^T\eta(\tau)d\tau \le 4C_K$ a.s., we have
\equal \label{Z-and-U-estimate}
\E\int_0^T \left({|Z_s|}^2 +\|U_s\|^2\right)ds  \leq 2 e^{4C_K}  \E |\xi|^2+ 2R e^{8C_K}   \E I_F^2+\frac{2}{R}\E\sup_{t\in{[0,T]}}|Y_t|^2.
\tionl
Now, we can plug in \eqref{Z-and-U-estimate} into  \eqref{supreme} and vice versa which yields for $R:= 48 c_1$ that
\begin{align*} 
\E\sup_{t\in{[0,T]}}|Y_t|^2\leq&\ (2c_1+48c_1e^{4C_K})\E|\xi|^2+ \left(2c_1+(48c_1)^2e^{8C_K}\right) \E  I_F^2,
\end{align*}
and
\equa
\E\!\!\int_0^T \!\!\!\left({|Z_s|}^2 +\|U_s\|^2\right)\!ds  \leq \!\left(\frac{1}{12} + 4 e^{4C_K}\right)\! \E |\xi|^2\!+\! \left(\frac{1}{12} + 192 c_1 e^{8C_K} \right)  \E I_F^2.
\tion
Using \eqref{the-constant} it is easy to see that there exists a constant $C_1>0$ such that each  factor in front of the expectations on the right 
side of the previous two 
inequalities  is less than
$e^{C_1(1+C_K)^2}.$

\end{proof}
Our next proposition will be an $L^2$ a-priori estimate  for BSDEs of our type. For the Brownian case, $L^p$ a-priori estimates are done for $p\in [1,\infty[$ 
in \cite{Briand-etal}, and for quadratic BSDEs, for $p \in [2,\infty[$ in \cite{Geiss}. For BSDEs with jumps, for  $p\in ]1,\infty[,$  see 
\cite{Kruse},\cite{Kruse17}; while \cite{Bech} contains an a-priori estimate w.r.t.~$L^\infty.$ The following assertion is similar to \cite[Proposition 2.2]{bbp}, but 
fits our extended setting. 
\begin{prop}\label{continuitythm}
Let $\xi,\xi'\in L^2$ and let $f,f'$ be two generator functions satisfying \ref{1}--\ref{3}, where the bounds in \ref{2} and the 
coefficients in \ref{3} may differ for $f$ and $f'$. The coefficients of $f'$ in \ref{3} will be referred to as 
$\alpha'$ and $\beta'$.
 Moreover, let the triplets $(Y,Z,U)$ and $(Y',Z',U')\in L^2(W)\times L^2(W)\times L^2(\tilde{N})$, satisfy the BSDEs $(\xi,f)$ and $(\xi',f'),$ 
 respectively. 

Then, 
\begin{align*}
&\|Y-Y'\|^2_{L^2(W)} +\left\|Z-Z'\right\|_{L^2(W) }^2 + \left\|U-U'\right\|_{L^2(\tilde N) }^2 \\
&\leq h\left(a,b,\E|\xi-\xi'|^2+2\E\!\int_0^T\!|Y_t-Y'_t|\left|f(t,Y_t,Z_t,U_t)-f'(t,Y_t,Z_t,U_t)\right|dt\right)\!,
\end{align*}
where  $a=\int_0^T\alpha'(s)ds,$  $b=\left\|\int_0^T\beta'(s)^2ds\right\|_\infty,$ and 
$$h:]0,\infty[\times]0,\infty[\times [0,\infty[\to [0,\infty[ $$ is a function such that 
$h(a,b,x)\to 0=h(a,b,0)$ if $x\to 0.$ 
\end{prop}
\begin{proof}
We start with the following observation gained by It\^o's formula for the difference of the BSDEs $(\xi,f)$ and $(\xi',f')$. We denote differences 
of expressions by $\Delta$. If  $\eta=4 \beta'(s)^2,$   we have analogously to \eqref{gammaeq1}
\begin{align} \label{delta-ito}
& e^{\int_0^t\eta(s)ds}|\Delta Y_t|^2+\int_t^T e^{\int_0^s\eta(\tau)d\tau}\left(\eta(s)|\Delta Y_s|^2 +|\Delta Z_s|^2+ \|\Delta U_s\|^2\right)ds \notag\\
&=e^{\int_0^T\eta(s)ds}|\Delta\xi|^2+M(t)  \notag \\
& \quad+\int_t^T 2e^{\int_0^s\eta(\tau)d\tau}\Delta Y_s \,(f(s,Y_s,Z_s,U_s)-f'(s,Y'_s,Z'_s,U'_s))ds,
\end{align}
where 
\begin{align*}
M(t) =& -\int_t^T 2e^{\int_0^s\eta(\tau)d\tau}\Delta Y_s\Delta Z_s dW_s \\
&-\int_{{]t,T]}\times\R_0}2e^{\int_0^s\eta(\tau)d\tau} \left((\Delta Y_{s-}+\Delta U_s(x))^2-\Delta Y_{s-}^2\right)\tilde{N}(ds,dx).
\end{align*}  By the same reasoning as for  \eqref{mart}, we have $\mathbb{E}M(t)=0$. 
We now proceed with the (standard) arguments similar to those used for \eqref{gammaeq1}--\eqref{gammaeq2}. By \ref{3}  and the first inequality 
from \eqref{algebra}, 
\begin{align} \label{R-ineq}
\Delta Y_s \,& (f'(s,Y_s,Z_s,U_s)-f'(s,Y'_s,Z'_s, U'_s)) \notag \\ 
\le & \alpha'(s)\rho(|\Delta Y_s|^2)+\beta'(s)|\Delta Y_s|(|\Delta Z_s|+\|\Delta U_s\|) \notag\\
\le & \alpha'(s)\rho(|\Delta Y_s|^2)+   \frac{ \beta'(s)^2|\Delta Y_s|^2}{R} +\frac{R (|\Delta Z_s|^2+\|\Delta U_s\|^2)}{2}. 
\end{align}
Taking the expectation in \eqref{delta-ito} and then using \eqref{R-ineq} with  $R=1$ (such that we can cancel out the terms with $Z$ and $U$  on the left  side), 
leads to 
\begin{align*}
& \E e^{\int_0^t\eta(s)ds}|\Delta Y_t|^2+\E \int_t^T e^{\int_0^s\eta(\tau)d\tau}\eta(s)|\Delta Y_s|^2ds\\
& \leq\E e^{\int_0^T\eta(s)ds}|\Delta\xi|^2 +\E\int_t^T 2 e^{\int_0^s\eta(\tau)d\tau}\Delta Y_s \cdot(\Delta f)(s,Y_s,Z_s,U_s)ds\\
& \quad\quad\quad\quad+\E\int_t^T e^{\int_0^s\eta(\tau)d\tau}\left(2\alpha'(s)\rho(|\Delta Y_s|^2)+ \beta'(s)^2|\Delta Y_s|^2\right)ds.
\end{align*}
The choice $\eta(s)=4 \beta'(s)^2$ and   the fact that  $ \int_0^T\beta'(s)^2ds \le b$  a.s.  leads to
\begin{align*}
 E|\Delta Y_t|^2\leq& e^{4b}\left(\E |\Delta\xi|^2 +\E\int_t^T 2|\Delta Y_s| |(\Delta f)(s,Y_s,Z_s,U_s)|ds\right)\\
&+e^{4b}\int_t^T 2\alpha'(s)\rho(\E|\Delta Y_s|^2)ds,
\end{align*}
since ${\rho}$ is a concave function.

By Proposition \ref{bihari-prop}, a backward version of the Bihari--LaSalle inequality, shows 
\begin{align}\label{bihari}
 &\sup_{t\in{[0,T]}}\E|\Delta Y_t|^2 \nonumber\\
 &\leq G^{-1}\Biggl \{G\bigg[e^{4b}\left(\E|\Delta\xi|^2 +\E\int_0^T 2|\Delta Y_s| |(\Delta f)(s,Y_s,Z_s,U_s)|ds\right)\bigg]\\
 &\quad\quad\quad\quad\quad\quad\quad\quad\quad\quad\quad\quad\quad+2e^{4b}\int_0^T \alpha'(s)ds\Biggr \},\nonumber
\end{align}
where $G(x)=\int_1^x\frac{1}{{\rho}(h)}dh.$

If we  take the expectation in \eqref{delta-ito}  but choose this time \eqref{R-ineq} with  $R=\frac{1}{2}$  and omit $\E e^{\int_0^t\eta(s)ds}|\Delta Y_t|^2,$ then 
\begin{align*}
& \E\int_t^T e^{\int_0^s\eta(\tau)d\tau}\left(\eta(s)|\Delta Y_s|^2+|\Delta Z_s|^2+\|\Delta U_s\|^2\right)ds\\
\leq \, & \E e^{\int_0^T\eta(s)ds}|\Delta\xi|^2 +\E\int_t^T 2 e^{\int_0^s\eta(\tau)d\tau}\Delta Y_s \cdot(\Delta f)(s,Y_s,Z_s,U_s)ds\\
&+\E \bigg \{\int_t^T\! e^{\int_0^s\eta(\tau)d\tau}\bigg(2\alpha'(s)\rho(|\Delta Y_s|^2)+ 4 \beta'(s)^2|\Delta Y_s|^2\\
&\quad\quad\quad\quad\quad\quad\quad\quad\quad\quad\quad\quad\quad\quad\quad\quad\quad\quad +\frac{|\Delta Z_s|^2+\|\Delta U_s\|^2}{2}\bigg)ds \bigg \}.
\end{align*}
We  subtract the quadratic terms with  $\Delta Y, \Delta Z,$ and $\Delta U$ which appear on the right hand side.  This results in the inequality 
\begin{align*}
& \E\int_t^T e^{\int_0^s\eta(\tau)d\tau}\left(|\Delta Z_s|^2+\|\Delta U_s\|^2\right)ds\\
&\leq 2\Biggl(\E e^{\int_0^T\eta(s)ds}|\Delta\xi|^2 +\E\int_t^T 2 e^{\int_0^s\eta(\tau)d\tau}|\Delta Y_s| \cdot|(\Delta f)(s,Y_s,Z_s,U_s)|ds\\
&\quad\quad+\E\int_t^T\! e^{\int_0^s\eta(\tau)d\tau}2\alpha'(s)\rho(|\Delta Y_s|^2))ds\Biggr).
\end{align*}
We  continue our estimate by 
\begin{align}\label{ZandUestim}
&\E\int_t^T e^{\int_0^s\eta(\tau)d\tau}\left(|\Delta Z_s|^2+\|\Delta U_s\|^2\right)ds \nonumber\\
&\leq 2e^{4b}\Bigg(\E|\Delta\xi|^2 +\E\int_t^T 2|\Delta Y_s| \cdot|(\Delta f)(s,Y_s,Z_s,U_s)|ds\\
&\quad +2\int_t^T \alpha'(s)ds \, {\rho}\left(\sup_{s\in{[0,T]}}\E|\Delta Y_s|^2\right)\Bigg),\nonumber
\end{align}
since $\eta(s)=4\beta'(s)^2$. 
We put 
\begin{align*}
 H:=&  G^{-1}\Biggl\{G\bigg[e^{4b}\left(\E|\Delta\xi|^2 +\E\int_0^T 2|\Delta Y_s| |(\Delta f)(s,Y_s,Z_s,U_s)|ds\right)\bigg] \\
 &+2e^{4b}\int_0^T \alpha'(s)ds\Biggr \}
\end{align*}
so that  \eqref{bihari} reads now as
 $ \sup_{t\in{[0,T]}}\E|\Delta Y_t|^2 \leq  H.$   If we add this inequality to  \eqref{ZandUestim} and note that 
 ${\rho}\left(\sup_{s\in{[0,T]}}\E|\Delta Y_s|^2\right)\le \rho(H),$ we have
\begin{align*}
&\hspace{-2em} \sup_{t\in{[0,T]}}\E|\Delta Y_t|^2+\E\int_0^T |\Delta Z_s|^2ds+\E\int_0^T \|\Delta U_s\|^2ds \\
\leq& 2e^{4b}\left(\E|\Delta\xi|^2+\E\int_0^T 2|\Delta Y_s| \cdot|(\Delta f)(s,Y_s,Z_s,U_s)|ds\right)\\
&+\left(2e^{4b}\int_0^T\alpha'(s)ds+1\right)\cdot (\mathrm{id}+{\rho})(H).
\end{align*}
Note that the integral condition on $\rho$ implies that, if the argument of $G$ approaches zero, then the right hand side vanishes.
\end{proof}

The following Lemma will be used to estimate the expectation of integrals which contain $|Y_s|^2.$

\begin{lemma} \label{lemma-eta-trick}
Let  $\xi\in L^2$ and assume that  \ref{1} and \ref{2} hold. If $(Y,Z,U)$ is a solution to  $(\xi,f)$ 
and $H$ is a nonnegative, progressively measurable process with $\left \|\int_0^T H(s)ds \right\|_\infty < \infty,$ 
then 
\begin{align}  \label{eta-trick}
\E \int_0^T H(s)  |Y_s|^2ds\leq \,&   e^{2C_K}   \E\int_0^T H(s)ds|\xi|^2\nonumber \\
&+ 2   e^{2C_K}  \left \|\int_0^T H(s)ds \cdot I_F\right\|_2\|Y\|_{\mathcal{S}^2}. 
\end{align}
\end{lemma}
\begin{proof}   

From the relations  \eqref{gammaeq1}, \eqref{mart} and  integration by parts applied to the term $\int_0^T H(s)ds \cdot   e^{\int_0^T \eta(s)ds}|Y_T|^2, $
we  get
\equa
&& \hspace{-1em}\int_0^T H(s)ds \cdot   e^{\int_0^T \eta(r)dr}|Y_T|^2 \\ 
&=& \int_0^T H(s)  e^{\int_0^s \eta(\tau)d\tau}|Y_s|^2 ds  -  \int_0^T \biggl(\int_0^s H(\tau)d\tau\biggr) \,\,dM(s)\\
&&+ \int_0^T \biggl(\int_0^s H(r)dr \biggr)  e^{\int_0^s\eta(\tau)d\tau} \Big (\eta(s){|Y_s|}^2 +{|Z_s|}^2 +\|U_s\|^2  \\
&& \hspace{18em} -2Y_sf(s,Y_s,Z_s,U_s) \Big)ds. 
\tion
We take expectations  and rearrange the equation so that 
\begin{align*}
&\hspace{-3em} \E\!\int_0^T\!  H(s) e^{\int_0^s \eta(\tau)d\tau} |Y_s|^2ds \\
\leq& \,\E\left[\int_0^T \! H(s)ds\cdot  e^{\int_0^T \eta(s)ds}|\xi|^2\right]\\
&+\E\biggl[\int_0^T\biggl(\int_0^sH(\tau)d\tau\biggr) e^{\int_0^s \eta(\tau)d\tau} \big( 2Y_sf(s,Y_s,Z_s,U_s) \\
&\hspace{8em} -\eta(s)|Y_s|^2-|Z_s|^2-\|U_s\|^2\big)ds\biggr].
\end{align*}
By assumption \ref{2} and \eqref{algebra}, we have
\equa
&&2Y_sf(s,Y_s,Z_s,U_s)  \\
&&\le 2 |Y_s| F(s)+2K_1(s)|Y_s|^2+ 2K_2(s) |Y_s| (|Z_s|+\|U_s\|) \\
&&\le 2 |Y_s| F(s)+2K_1(s)|Y_s|^2+ 2 K_2(s)^2  |Y_s|^2 +  |Z_s|^2+\|U_s\|^2, 
\tion
so that for $\eta(s)=2K_1(s) + 2 K_2(s)^2$ it follows
\begin{align} \label{eta-half-trick}
\E\!\int_0^T\! H(s)  |Y_s|^2ds
\leq \,&   \E\biggl[\int_0^T H(s)ds\cdot  e^{\int_0^T \eta(s)ds}|\xi|^2\biggr]\nonumber\\
&+ 2\E\biggl[\int_0^T\biggl(\int_0^sH(\tau)d\tau\biggr) e^{\int_0^s \eta(\tau)d\tau}F(s)|Y_s|ds\biggr] \nonumber\\
\leq \,& e^{2C_K}   \E\biggl[\int_0^T H(s)ds \cdot |\xi|^2\biggr]\nonumber \\
&+ 2   e^{2C_K}  \left \|\int_0^T H(s)ds\cdot I_F\right\|_2\|Y\|_{\mathcal{S}^2}.
\end{align}
\end{proof}

\section{Proofs of Theorems \ref{existence}    and \ref{finlev}} \label{sec5}
\subsection{Proof of Theorem \ref{existence}}
\label{4-2}
\texttt{Step 1:} Uniqueness\\
Uniqueness of the solution is a consequence of Proposition \ref{continuitythm}, since the terms $|\xi-\xi'|$ and $|f(s,Y_s,Z_s,U_s)-f'(s,Y_s,Z_s,U_s)|$ are zero.
\bigskip

The proof of existence will be split up in further steps.

\texttt{Step 2:} \\
In this step, we construct an approximating sequence of generators $f^{(n)}$ for $f$ and show several estimates for the solution processes 
$(Y^{n},Z^{n}, U^{n})$ to the BSDEs $(\xi,f^{(n)})$.\bigskip

For $n \ge 1,$ define $c_n(z):=\min(\max(-n,z),n)$ and $\tilde{c}_n(u)\in L^2(\nu)$ to be the projection of $u$ onto $\{v\in{L^2(\nu)}:\|v\| \leq n\}$. 
Let $(Y^{n},Z^{n}, U^{n})$ be the unique solution of the BSDE $(\xi,f^{(n)})$, with the definitions 
$$\hat{f}^{(n)}(\omega,s,y,z,u):=f(\omega,s,y,c_n(z),\tilde c_n(u)),$$ 
and 
\begin{align*}
&f^{(n)}(\omega,s,y,z,u):=\mathrm{sign}\left(\hat{f}^{(n)}(\omega,s,y,z,u)\right)\\
&\times\left [ F(\omega, s)\wedge n +(K_1(\omega,s)\wedge n)|y| +(K_2(\omega,s)\wedge n)(|c_n(z)|+\left\|\tilde{c}_n(u)\right\|)\right]
\end{align*}  
\begin{align*}  \text{ if }  \quad 
|\hat{f}^{(n)}(\omega,s,y,z,u)|>\ &F(\omega, s)\wedge n
+(K_1(\omega,s)\wedge n)|y|\\  \nopagebreak[4]
&+(K_2(\omega,s)\wedge n)(|c_n(z)|+\|\tilde{c}_n(u)\|),
\end{align*}
and $$f^{(n)}(\omega,s,y,z,u):=\hat{f}^{(n)}(\omega,s,y,z,u)  \quad \text{else.}$$
{\ch Note that $f^{(n)}$ satisfies \ref{1}--\ref{4}, with the same coefficients as $f.$ 
Moreover, by \ref{4}, $f^{(n)}$  satisfies a Lipschitz condition with respect to $u$ (see Remark \ref{remark}).   
 Thus,  thanks to \cite[Theorem 2.1]{YinMao}, $(\xi,f^{(n)})$ has a unique solution  $(Y^n,Z^n,U^n).$ 
Moreover,} by Proposition \ref{supprop}, we get that 
\equal \label{boundedYZU}
  \|Y^n\|^2_{\mathcal{S}^2} +\left\|Z^n\right\|_{L^2(W) }^2 + \left\|U^n\right\|_{L^2(\tilde N) }^2 
 \leq e^{C_1(1+ C_K)^2} (\E|\xi|^2+\E I_F^2)<\infty,
\tionl
uniformly in $n.$
This  implies that the families \begin{samepage} $$\left({\sup_{t\in{[0,T]}}|Y_t^n|},n\geq 0\right), \left({|Y^n|},n\geq 0\right)\text{ and }\left(|Z^n|+\|U^n\|,n\geq 0\right)\nopagebreak[4] $$  are uniformly integrable with respect to $\mathbb{P}$, $\mathbb{P}\otimes\lambda$ and $\mathbb{P}\otimes\lambda$, respectively. 
\end{samepage}

\texttt{Step 3:}\\
The goal of this step is to use Proposition \ref{continuitythm} to get convergence of $(Y^n,Z^n,U^n)_n$ in $L^2(W)\times L^2(W)\times L^2(\tilde N)$ 
for a subsequence $n_k \uparrow \infty$  if  $\delta_{n_k,n_l} \to 0$ for $k>l \to\infty,$  where 
\equa
\delta_{n,m}&:=&  \E\int_0^T |Y^n_s-Y^m_s||f^{(n)}(s,Y^n_s,Z^n_s,U^n_s)-f^{(m)}(s,Y^n_s,Z^n_s,U^n_s)|ds.\\
\tion 
We observe that the difference of the generators is zero if two conditions are satisfied at the same time: First, if $|Z^n|,\|U^n_s\|<n$, and additionally, by the cut-off procedure for ${F}, K_1, K_2$, if 
$$n>\max\left({F}(\omega,s),K_1(\omega,s), K_2(\omega,s)\right)=:k(\omega,s).$$ Thus, putting 
\equal \label{chi}
\chi_n (s) :=
\chi_{\{ |Z^n_s|> n\} \cup \{\|U^n_s\| >n\} \cup \{k(s)>n \}},
\tionl
 we have
\begin{align*} 
\delta_{n,m} &= \E\int_0^T |Y^n_s-Y^m_s||f^{(n)}(s,Y^n_s,Z^n_s,U^n_s)-f^{(m)}(s,Y^n_s,Z^n_s,U^n_s)|\chi_n(s)ds\\
&\le \E \bigg \{\int_0^T 2|Y^n_s-Y^m_s|\, \chi_n(s)\\
&\hspace{6em} \times\left({F}(s)+K_1(s)|Y^n_s|+K_2(s)(|Z^n_s|+\|U^n_s\|)\right)ds \bigg \},\nonumber
\end{align*}
due to the linear growth condition \ref{2}.  We estimate this further by 
\begin{align} \label{A2-used} 
\delta_{n,m} &\le   \E\int_0^T\chi_n(s)\,  F(s)ds  \,  \left(\sup_{r\in{[0,T]}}| Y^n_r|+\sup_{r\in{[0,T]}}|Y^m_r|\right) \nonumber \\
&\quad +   \E\int_0^T\chi_n(s)\,K_2(s)(| Y^n_s|+|Y^m_s|)(|Z^n_s|+\|U^n_s\|)ds   \nonumber \\
&\quad +\E\int_0^T 2|Y^n_s-Y^m_s|\, |Y^n_s|   \chi_n(s) \,K_1(s)ds \nonumber\\
&=: \delta^{(1)}_{n,m} + \delta^{(2)}_{n,m}+ \delta^{(3)}_{n,m}.
\end{align}
For $\delta^{(1)}_{n,m},$  we use the Cauchy--Schwarz inequality,
$$ \delta^{(1)}_{n,m} \le 2 \left ( \E \left |\int_0^T\chi_n(s)\,  F(s)ds\right |^2 \right )^\frac{1}{2}   (\|Y^n\|_{\mathcal{S}^2}+\|Y^m\|_{\mathcal{S}^2}). $$
Since 
 $\sup_n \|Y^n\|_{\mathcal{S}^2}< \infty$  according to \eqref{boundedYZU}, it remains to show that 
the integral term converges to $0$ for a subsequence.

Since $|Z^n_s|$ and $\|U^n_s\|$ are uniformly integrable w.r.t. $\mathbb{P}\otimes\lambda,$
we imply from \eqref{chi}  that $\chi_n \to 0$ in $L^1(\mathbb{P}\otimes\lambda).$  Hence, there exists a subsequence 
$(n_k)_{k\ge 1}$ such that 
\equal \label{chi-to-0}
\chi_{n_k} \to 0  \quad k\to\infty, \quad\mathbb{P}\otimes\lambda \text{-a.e.} 
\tionl
 By dominated convergence, 
we have $ \E\left | \int_0^T\chi_{n_k}(s) F(s)ds \right |^2 \to 0 $ for $k\to\infty$    since $F\in L^2(\Omega; L^1([0,T])).$  \\
For $ \delta^{(2)}_{n,m},$ we start with the Cauchy--Schwarz inequality and get 
 \equa 
\delta^{(2)}_{n,m}   &\le& 2 \sup_k \left [ \left\|Z^k\right\|_{L^2(W) } + \left\|U^k\right\|_{L^2(\tilde N) }\right] \\
&&\times \left [ \E\int_0^T\chi_n(s)\,K_2(s)^2 (| Y^n_s|^2 +|Y^m_s|^2)ds\right]^\frac{1}{2}.  
\tion
By Lemma \ref{lemma-eta-trick}, 
\equal \label{eta-trick-used}
&& \E\int_0^T\chi_n(s)\,K_2(s)^2 (| Y^n_s|^2 +|Y^m_s|^2)ds \nonumber \\
&& \quad\le 2e^{2C_K}   \E\int_0^T\chi_n(s)\,K_2(s)^2ds |\xi|^2  \\
&&\quad+ 2e^{2C_K}\left\|\int_0^T\chi_n(s)\,K_2(s)^2ds \cdot I_F \right\|_2( \|Y^n\|_{\mathcal{S}^2}+ \|Y^m\|_{\mathcal{S}^2} ). \notag
\tionl
Hence,  \eqref{chi-to-0} implies $\delta^{(2)}_{n_k,m} \to 0$ for $k\to \infty.$\\
Finally,
\equa
\delta^{(3)}_{n,m} \le 2\E\int_0^T (2|Y^n_s|^2 + |Y^m_s|^2 )\,   \chi_n(s) \,K_1(s)ds,
 \tion
so that we can argue like in \eqref{eta-trick-used} to get that  $\delta^{(3)}_{n_k,m} \to 0$ for $k\to \infty.$\\
 Thus  $(Y^{n_k},Z^{n_k},U^{n_k})_{k\ge1}$ converges to an object $(Y,Z,U)$ in $L^2(W)\times L^2(W)\times L^2(\tilde{N})$.\bigskip

\texttt{Step 4:}\\
In the final step, we want to show that $(Y,Z,U)$ solves $(\xi,f)$. For the approximating sequence $(Y^{n_k},Z^{n_k},U^{n_k})_{k\ge1},$
the stochastic integrals and the left hand side of the BSDEs $(\xi,f^{(n_k)})$ obviously converge in $L^2$ to the corresponding terms of $(\xi,f)$. Therefore, this 
subsequence of $(\int_t^T f^{(n)}(s,Y^{n}_s,Z^{n}_s,U^{n}_s)ds)_{n=1}^\infty$ converges to a random variable $V_t$. We need to show that $V_t=\int_t^T f(s,Y_s,Z_s,U_s)ds$. To achieve this, consider
\begin{align}\label{differences}
 \delta_n:=&\E  \int_t^T |f^{(n)}(s,Y^{n}_s,Z^{n}_s,U^{n}_s)  - f(s,Y^{n}_s,Z^{n}_s,U^{n}_s)|ds \notag\\
&+ \E\int_t^T| f(s,Y^{n}_s,Z^{n}_s,U^{n}_s)-f(s,Y_s,Z_s,U_s)|ds.
\end{align}
We start with the first integrand where, by the definition of $f_n$ and \eqref{chi}, and the growth condition \ref{2}, 
\equa
&& |f^{(n)}(s,Y^{n}_s,Z^{n}_s,U^{n}_s)  - f(s,Y^{n}_s,Z^{n}_s,U^{n}_s)| \\
&&\quad = |f^{(n)}(s,Y^{n}_s,Z^{n}_s,U^{n}_s)  - f(s,Y^{n}_s,Z^{n}_s,U^{n}_s)| \chi_n \\
&&\quad \le 2\left( F(s) \chi_n(s) +K_1(s)|Y^{n}_s| \chi_n(s) +K_2(s)\chi_n(s)(|Z^{n}_s|+\|U^{n}_s\|)\right)\\
&&\quad =: 2(\kappa^{(1)}_n(s) +\kappa^{(2)}_n(s) +\kappa^{(3)}_n(s)).
\tion 
The estimates are similar as in the previous step. Thanks to \eqref{chi-to-0}, we have \\
$\E  \int_t^T\kappa^{(1)}_{n_k}(s)ds \to 0.$
For the next term, the Cauchy--Schwarz inequality yields
\equa
\E  \int_t^T\kappa^{(2)}_{n_k}(s)ds  \le  \left    \|\int_0^T\chi_n(s)\,K_1(s)ds \right\|_2   \sup_l \|Y^l\|_{\mathcal{S}^2},
\tion
so that  by \eqref{chi-to-0} the first factor converges to zero along  the subsequence $(n_k).$ 
The last term we estimate using the Cauchy--Schwarz inequality  w.r.t. $\PP\otimes \lambda,$
\equa
\E \! \int_t^T\!\kappa^{(3)}_{n_k}(s)ds  \le  \left[  \E \int_0^T\!\!K_2(s)^2\chi_n(s) ds \right ]^\frac{1}{2}     \sup_l \left [ \left\|Z^l\right\|_{L^2(W) } \!\!+ \left\|U^l\right\|_{L^2(\tilde N) }\right],
\tion
 and again  by \eqref{chi-to-0},  we have convergence to zero along  the subsequence $(n_k).$ 
 
We continue  showing the convergence of the second term in \eqref{differences}. We extract a sub-subsequence of $(n_{k})_{k\geq 1}$, which we call---slightly abusing the notation---again $(n_{k})_{k\geq 1}$ such that $(Y^{n_{k}},Z^{n_{k}},U^{n_{k}})$, regarded as a triplet of measurable functions with values in $\R\times\R\times L^2(\nu)$, converges to $(Y,Z,U)$ for $\mathbb{P}\otimes\lambda$-a.a. $(\omega,s)$ .
Then, for an arbitrary $K>0$, we have  
\begin{align} \label{KO}
&\E\int_t^T \left|f(s,Y^{n_{k}}_s,Z^{n_{k}}_s,U^{n_{k}}_s)-f(s,Y_s,Z_s,U_s)\right|ds \notag \\
&\le \E \bigg \{\int_t^T \left|f(s,Y^{n_{k}}_s,Z^{n_{k}}_s,U^{n_{k}}_s)-f(s,Y_s,Z_s,U_s)\right|\\\nonumber
&\quad\times\left(\chi_{\{|Y^{n_{k}}_s|\leq K, |Z^{n_{k}}_s|+\|U^{n_{k}}_s\|\leq K\}}+\chi_{\{|Y^{n_{k}}_s|>K\}}+\chi_{\{ |Z^{n_{k}}_s|+\|U^{n_{k}}_s\|>K\}}\right)ds \bigg \}.\nonumber
\end{align}
By dominated convergence and the continuity of $f,$  
\begin{align*}
\E\!\int_t^T \!\!\left|f(s,Y^{n_{k}}_s,Z^{n_{k}}_s,U^{n_{k}}_s)-f(s,Y_s,Z_s,U_s)\right|\!\chi\!_{\{|Y^{n_{k}}_s|\leq K, |Z^{n_{k}}_s|+\|U^{n_{k}}_s\|\leq K\}}ds \to 0,
\end{align*}
since by \ref{2} we can bound the integrand by $$2F(s)+K_1(s)(K+|Y_s|)+K_2(s)(2K+|Z_s|+\|U_s\|),$$ 
which is integrable. We let 
$$\chi_K(n_{k},s) := \chi_{\{|Y^{n_{k}}_s|> K\}}+\chi_{\{|Z^{n_{k}}_s|+\|U^{n_{k}}_s\| > K\}}.$$ 
Then, the remaining terms of \eqref{KO} are bounded by
\begin{align*}
& \E\int_0^T\left(2F(s)+K_1(s)|Y_s|+K_2(s)(|Z_s|+\|U_s\|)\right) \chi_K(n_{k},s) ds  \\
&+\E\int_0^TK_1(s)|Y^{n_{k}}_s| \chi_K(n_{k},s)  ds \\
&+\E\int_0^T K_2(s)(|Z^{n_{k}}_s|+\|U^{n_{k}}_s\|)\chi_K(n_{k},s)ds \\
=: &\, \delta^{(1)}_{n_k} + \delta^{(2)}_{n_k}+ \delta^{(3)}_{n_k}.
\end{align*}
If we choose a $K$ large enough, then $\delta^{(1)}_{n_k}$  can be made arbitrarily small since the families $\left(|Y_s^n|,n\geq 0\right)$ and $\left(|Z^{n}_s|+\|U^{n}_s\|,n\geq 0\right)$ are uniformly integrable with respect to $\mathbb{P}\otimes \lambda$. 
The same holds for
 \equa
(\delta^{(2)}_{n_k})^2 &\le& \E\left |\int_0^TK_1(s)\chi_K(n_{k},s)ds\right |^2 \,\sup_l \|Y^{n_l}\|_{\mathcal{S}^2}^2\\
&\le&\left \|\int_0^TK_1(s)ds\right \|_\infty\E\left[ \int_0^TK_1(s)\chi_K(n_{k},s)ds\right] \,\sup_l \|Y^{n_l}\|_{\mathcal{S}^2}^2,
\tion 
and
\equa
(\delta^{(3)}_{n_k})^2 \le 2\E\left[\int_0^TK_2(s)^2\chi_K(n_{k},s)ds\right] \,\sup_l \E\int_0^T(|Z^{n_{l}}_s|^2+\|U^{n_{l}}_s\|^2)ds.
\tion
Hence, for $\delta_n$ defined in \eqref{differences}, we have that $\lim_{k\to \infty} \delta_{n_k}=0,$ which implies 
$$\lim_{k\to\infty}\E\left|\int_t^T f^{(n_{k})}(s,Y^{n_{k}}_s,Z^{n_{k}}_s,U^{n_{k}}_s)ds-\int_t^T f(s,Y_s,Z_s,U_s)ds\right|=0.$$
We infer that for a sub-subsequence $(n_{k_{l}},l\geq 0)$ we get the a.s. convergence
$$\int_t^T f^{(n_{k_{l}})}(s,Y^{n_{k_{l}}}_s,Z^{n_{k_{l}}}_s,U^{n_{k_{l}}}_s)ds\to\int_t^T f(s,Y_s,Z_s,U_s)ds.$$
Thus, for the original sequence, a.s.
$$\int_t^T f^{(n_k)}(s,Y^{n_k}_s,Z^{n_k}_s,U^{n_k}_s)ds\to V_t=\int_t^T f(s,Y_s,Z_s,U_s)ds,$$
 and therefore the triplet $(Y,Z,U)$ satisfies the BSDE $(\xi,f)$.   
\qed\bigskip

\subsection{Proof of Theorem \ref{finlev}} 
\label{5-2}
We start with a preparatory lemma:
\begin{lemma}\label{optcont}
{\ch If $f$ satisfies \ref{1}--\ref{4}, then for all $n\geq 0$, $f_n$ constructed in Definition \ref{optional-f} also satisfies \ref{1}--\ref{4} (with different coefficients). }
\end{lemma}
\begin{proof}
By definition, $(\omega,t)\mapsto f_n(t,y,z,u)$ is progressively measurable for all $(y,z,u)$, thus \ref{1} is satisfied. 
The inequalities in \ref{2} and \ref{3} are a.s. satisfied, with coefficients $\E_n F, \E_n K_1, \E_n K_2, \E_n\beta$. To ensure that these coefficients have a  $\left(\mathcal{F}_t^n\right)_{t\in{[0,T]}}$-progressively measurable  version, one  applies the procedure from Definition  \ref{optional-f}  to the inequalities in \ref{2} and \ref{3}  and notes that an equation analogous  to \eqref{optionalcond} holds true. 

\bigskip

 It remains to show a.s.~continuity of $f_n$ in the $(y,z,u)$-variables required  in \ref{3}  for a.e.~$t.$  In \cite[Proposition 7.3]{Ylinen}, this  was shown by the fact that the 
approximation of the generators appearing there can be done using  spaces of continuous functions. However, since our situation involves $L^2(\nu)$, a non-locally 
compact space, we can not easily adapt the proof from \cite{Ylinen} and therefore we will use different means.

Let $\mathrm{D}[0,T]$ be the space of c\`adl\`ag functions endowed by the Skorohod metric (which makes this space a Polish space). The Borel $\sigma$-algebra $\mathcal{B}(\mathrm{D}[0,T])$ is generated by the coordinate projections $p_t\colon\mathrm{D}[0,T]\to\R, \mathrm{x}\mapsto\mathrm{x}(s)$ (see Theorem 12.5 of \cite{Billing}, for instance). On this $\sigma$-algebra, let $\mathbb{P}_X$ be the image measure induced by the L\'evy process $X$: $\Omega\to\mathrm{D}[0,T], \omega\mapsto X(\omega)$. We denote by $\mathcal{G}$ the completion with respect to $\mathbb{P}_X$.  For $t\in{[0,T]},$ the 
notation $$\mathrm{x}^t(s):=\mathrm{x}(t\wedge s),\text{ for all } s\in{[0,T]}$$ induces the natural identification $$\mathrm{D}{[0,t]}=\left\{\mathrm x\in \mathrm{D}{[0,T]} : \mathrm{x}^t=\mathrm{x} \right\}.$$ By this identification, 
we define a filtration on this space through
\equa 
\mathcal{G}_t=\sigma\left(\mathcal{B}\left(\mathrm{D}{[0,t]}\right)\cup \mathcal{N}_X{[0,T]}\right), \quad 0\leq t\leq T,
\tion where $\mathcal{N}_X{[0,T]}$ denotes the 
null sets of $\mathcal{B}\left(\mathrm{D}{[0,T]}\right)$ with respect to the  image measure $\mathbb{P}_X$  of the L\'evy process $X$. 
The same procedure applied to the L\'evy process $X^n$ yields a filtration $(\mathcal{G}_t^n)_{t\in{[0,T]}}$ defined in the same way.

According to \cite[Theorem 3.4]{SteinickeI}, which is a generalization of Doob's factorization lemma to random variables depending on  parameters,  there is a 
$\mathcal{G}_t\otimes\mathcal{B}([0,t]\times\R^2\times L^2(\nu))$-measurable functional 
$$g_{f}\colon \mathrm{D}[0,t]\times [0,t]\times\R^2\times L^2(\nu)\to\R$$ 
and a $\mathcal{G}_t^n\otimes\mathcal{B}([0, t]\times\R^2\times L^2(\nu))$-measurable functional 
$$g_{f_n}\colon \mathrm{D}[0,t]\times [0, t]\times\R^2\times L^2(\nu)\to\R$$ 
such that $\mathbb{P}$-a.s., 
\equal \label{factor}
g_f(X(\omega),\cdot)=f(\omega,\cdot) \quad \text{and} \quad g_{f_n}(X^n(\omega),\cdot)=f_n(\omega,\cdot).
\tionl Note also, that if $\mathbb{P}_X(M)=0$ for $M\in \mathcal{G}$, then also $\mathbb{P}(X^{-1}(M))=0$. Thus, without loss of generality, we may assume that $(\Omega,\mathcal{F},\mathbb{P})=(\mathrm{D}[0,T],\mathcal{G},\mathbb{P}_X)$ and $(\Omega,\mathcal{F}_t^n,\mathbb{P})=(\mathrm{D}([0,t]),\mathcal{G}_t^n,\mathbb{P}_X)$, which are standard Borel spaces. For more details on $\mathrm{D}{[0,T]}$, see \cite{Billing} and \cite[Section 4]{delzeith}.\bigskip

Now, fix $N\in\mathbb{N}$ and let $c_0:=\{(a_n)_n\in (\R^2\times L^2(\nu))^\mathbb{N}: a_n\to 0\}.$ For $a\in c_0$, let 
$\|a\|_{c_0}=\sup_{n\in\mathbb{N}} (|a_n(1)|+|a_n(2)|+\|a_n(3)\|)$, where $a(k),k=1,2,3$ are the components of $a$ in $\R$, $\R$ and $L^2(\nu)$.
The space $c_0$ is a Polish space.
Let $B_N$ be the ball with radius $N\in\mathbb{N}$ in $c_0$ and let $B'_N$ be the ball of  radius $N$ in $\R^2\times L^2(\nu)$. The balls $B_N, B'_N$ are again Polish spaces.
 
We consider a Borel set $M_T$ of $t\in [0,T]$ for which $f$ is continuous in $(y,z,u)$  and for which it holds that
$f$ has an integrable bound:
\equal \label{f-int}
  \E |f(t,y,z,u)| \le  \E F(t) +  \E K_1(t)|y|+ \E K_2(t)(|z|+\|u\|) < \infty.  
\tionl
From   \ref{3}  and \eqref{f-integrable} it follows that one can choose $M_T$ such that $\lambda(M_T)=T.$   \smallskip \\
For a fixed  $t \in M_T$ we define the function 
$$H_m:\Omega\times B_N\times B'_N\to\R, \, (\omega,a,\varphi)\mapsto f_n(\omega,t,a_m+\varphi),$$
where $\varphi$ denotes a triplet $(y,z,u)\in\R^2\times L^2(\nu)$.
This function is measurable since 
$f_n(\cdot,t,\cdot)$ is measurable, $\pi_m:B_N\times B'_N\to\R^2\times L^2(\nu), (a,\varphi)\mapsto(a_m+\varphi)$ is continuous and $\mathrm{id}\times\pi_m:\Omega\times B_N\times B'_N\to\Omega\times\R^2\times L^2(\nu)$ is measurable.

Next, we consider the map
$$H:\Omega\times B_N\times B'_N\to\R, (\omega,a,\varphi)\mapsto \begin{cases} \lim_{m\to\infty}f_n(\omega,t,a_m+\varphi), &\text{ if it exists}\\
0, & \text{ else}.\end{cases}$$
The set, where the limit exists is measurable, since it can be written as 
$$\bigcap_{k\geq 1}\bigcup_{M\geq 0}\bigcap_{m_1,m_2\geq M}\left\{|H_{m_1}-H_{m_2}|<\frac{1}{k}\right\}.$$
Therefore, $H$ can be written as the pointwise limit of measurable functions and is thus measurable.

We now know that, for a fixed pair $(a,\varphi)\in B_N\times B'_N$, $$f_n(t,a_m+\varphi)=\E_n f(t,a_m+\varphi), \quad\mathbb{P}\text{-a.s.}$$
Thus, by  \eqref{f-int} 
$$|f_n(t,a_m+\varphi)|\leq \E_n F(t)+2N\E_n K_1(t)+4N\E_n K_2(t) < \infty.$$ 
By the continuity of $f$ and the dominated convergence theorem for conditional expectations, we infer that up to a null set $M(a,\varphi)\in\mathcal{F}^n_t$, we have the relation
\begin{align*}
&\lim_{m\to\infty}f_n(t,a_m+\varphi)=\lim_{m\to\infty}\E_n f(t,a_m+\varphi)=\E_n \lim_{m\to\infty} f(t,a_m+\varphi)\\
&=\E_n f(t,\varphi)=f_n(t,\varphi).
\end{align*}
In other words, on the complement of $M(a,\varphi)$, we have $H(\omega,a,\varphi)=f_n(\omega, t,\varphi)$. This means that $H$ and $f_n(\cdot,t,\cdot)$ are ''versions'' of each other. What we need is ''indistinguishability'' of the processes.

For this purpose, let $(A,\Phi):\Omega\to B_N\times B'_N$ be an arbitrary $\mathcal{F}^n_t$-measurable function. 
Like above, by the definition of the optional projection, \ref{2}, and the continuity of $f$, we get the equation
\begin{align*}
&\lim_{m\to\infty}f_n(t,A_m+\Phi)=f_n(t,\Phi),
\end{align*}
which is also satisfied $\mathbb{P}$-a.s. This equality means, that $$H(\omega,A(\omega),\Phi(\omega))= f_n( \omega, t,\Phi(\omega)),\quad a.s.$$

All $\mathcal{F}^n_t$ were complete $\sigma$-algebras (in fact they contain all null sets of $\mathcal{F}$) and the spaces $B_N,B'_N$ were Polish. Thus we may use a generalized 
version of the section theorem, the Jankov--von Neumann theorem (Theorem \ref{JvN}),  by choosing a uniformizing function $(\hat A,\hat \Phi)$ for the set 
$$P=\{(\omega,a,\varphi): H(\omega,a,\varphi)\neq f_n(\omega,  t, \varphi)\}.$$
Note that $P$ is a Borel set and therefore especially analytic, since $H$ and $f_n(\cdot,t,\cdot)$ (interpreted as a  constant  map w.r.t.~$a$) are measurable functions in 
$(\omega, a,\varphi ).$ 
Since for this choice of $(\hat A,\hat \Phi)$ it holds, as seen above, that $$H(\omega,\hat A(\omega),\hat \Phi(\omega))= f_n(\omega,t,\hat \Phi(\omega)), \quad \text{  a.s. } $$ 
it follows that the projection of $P$ to $\Omega$ is a null set. Therefore, $H$ and $f_n$ are indistinguishable. Hence, we find a null set $M_N\in\mathcal{F}^n_t$, such that for $\omega$ outside this set and for all $(a,\varphi)\in B_N\times B'_N$: $$\lim_{m\to\infty}f_n(\omega,t,a_m+\varphi)=f_n(\omega,t,\varphi).$$
But this means continuity in all points of $B'_N$ a.s. 
It remains to unite the sets $M_N$ for all $N\in\mathbb{N}$, to obtain a set such that on its complement the function is continuous in all points of $\R^2\times L^2(\nu)$.
\end{proof}
\renewcommand*{\proofname}{Proof of Theorem \ref{finlev}}
\begin{proof}
\texttt{Step 1:}\\
If $f$ satisfies \ref{1}--\ref{4}, by Lemma \ref{optcont} all $f_n$ do so as well. In this case, for all $n\geq 0$, the equations $(\E_n\xi,f_n)$ have unique solutions  
 by Theorem \ref{existence}. In general, the coefficients in \ref{2} 
and $\beta$ differ dependent on $n$ since $F, K_1, K_2, \beta$ will be replaced by the coefficients $\E_n F, \E_n K_1, \E_n K_2,  \E_n\beta$.\\
Let us compare the solutions $(Y^n,Z^n,U^n)$ and $(Y,Z,U).$
We start comparing $(Y^n,Z^n,U^n)$ and $(\E_n Y, \E_n Z, \E_n U)$. Here, for instance, the process $((\E_n Y)_t)_{t\in{[0,T]}}$  is defined as an optional projection with respect to the filtration 
$\left(\mathcal{F}^n_t\right)_{t\in{[0,T]}}$, similar to Definition \ref{optional-f}. The so defined processes are versions of the processes $(\E_n Y_t,\E_n Z_t, \E_n U_t)_{t\in{[0,T]}}$.     

Using the BSDE for $(Y,Z,U)$, we get $\mathbb{P}$-a.s.
\begin{align}\label{BSDE-conditional}
\E_n Y_t=& \,\E_n \xi+\int_t^T\E_n f(s,Y_s,Z_s,U_s)ds - \int_t^T\E_n Z_s dW_s\nonumber\\
&-\int_{{]t,T]}\times\{1/n \leq |x|\}}\E_n U_s(x)\tilde{N}(ds,dx),
\end{align}
since  $$\E_n  \int_{{]t,T]}\times\{1/n > |x|\}} U_s(x)\tilde{N}(ds,dx)=0. $$  
Now, to estimate $\|Y^n-\E_n Y\|_{L^2(W)}+\|Z^n-\E_n Z\|_{L^2(W)}+\|U^n-\E_n U\|_{L^2(\tilde N)}$, we apply It\^o's formula to the difference of the BSDE $(\E_n \xi,f_n)$ and \eqref{BSDE-conditional}. Similar to the proof of Proposition \ref{continuitythm}, we get, denoting differences by $\Delta^n$ and $\eta:=4\beta(s)^2,$
\begin{align*}
& \E \bigg \{e^{\int_0^t\eta(s)ds}|\Delta^n Y_t|^2+\int_t^T e^{\int_0^s\eta(\tau)d\tau}\left(\eta(s)|\Delta^n Y_s|^2 +|\Delta^n Z_s|^2+ \|\Delta^n U_s\|^2\right)ds\bigg \} \notag\\
&=\E\int_t^T 2e^{\int_0^s\eta(\tau)d\tau}(\Delta^n Y_s) \,(f_n(s,Y^n_s,Z^n_s,U^n_s)-\E_n f(s,Y_s,Z_s,U_s))ds.
\end{align*}
By the measurability of $(Y^n,Z^n,U^n)$, the equality $$f_n(s,Y^n_s,Z^n_s,U^n_s)=\E_n f(s,Y^n_s,Z^n_s,U^n_s)$$ holds $\mathbb{P}$-a.s. for all $s$. We now estimate
\begin{align*}
&\hspace{-3em} \E\Big[(\Delta^n Y_s) \,(f_n(s,Y^n_s,Z^n_s,U^n_s)-\E_n f(s,Y_s,Z_s,U_s))\Big]\\
= \,&\E\Big[(\Delta^n Y_s) \,(\E_n f(s,Y^n_s,Z^n_s,U^n_s)-\E_n f(s,Y_s,Z_s,U_s))\Big]\\
= \,&\E\Big[(\Delta^n Y_s) \,(f(s,Y^n_s,Z^n_s,U^n_s)-f(s,Y_s,Z_s,U_s))\Big]\\
=\,&\E\Big[(\Delta^n Y_s) \,(f(s,Y^n_s,Z^n_s,U^n_s)-f(s,\E_n Y_s,\E_n Z_s,\E_n U_s))\Big]\\
&+\E\Big[(\Delta^n Y_s) \,(f(s,\E_n Y_s,\E_n Z_s,\E_n U_s)-f(s,Y_s,Z_s,U_s))\Big]\\
\leq\,& \E\Big[\alpha(s)\rho(|\Delta^n Y_s|^2)+\beta(s)|\Delta^n Y_s|(|\Delta^n Z_s|+\|\Delta^n U_s\|)\Big]\\
&+\E\Big[|\Delta Y^n_s||(f(s,\E_n Y_s,\E_n Z_s,\E_n U_s)-f(s,Y_s,Z_s,U_s)|\Big].
\end{align*}
Now, we can conduct exactly the same steps as in the standard procedure used in the proof of Proposition \ref{continuitythm}. This means that  $\|\Delta^n Y\|_{L^2(W)}+\|\Delta^n Z\|_{L^2(W)}$ $+\|\Delta^n U\|_{L^2(\tilde N)}$ converges to zero if
\begin{align}\label{want}
\E\int_0^T |\Delta Y^n_s||(f(s,\E_n Y_s,\E_n Z_s,\E_n U_s)-f(s,Y_s,Z_s,U_s)|ds
\end{align}
does, which we will show in the following steps.

\texttt{Step 2:}\\
In this step, we show that the solution processes $(Y^n,Z^n, U^n)$ satisfy the estimate
\begin{equation}\label{supns}
\sup_{n\geq 0} \left(\|Y^n\|_{\mathcal{S}^2} +\left\|Z^n\right\|_{L^2(W) }^2 + \left\|U^n\right\|_{L^2(\tilde N) }^2\right)<\infty.
\end{equation}
This, as in the proof of Theorem \ref{existence}, leads to the uniform integrability of the processes 
$\left({|Y^n|},n\geq 0\right)\text{ and }\left(|Z^n|+\|U^n\|,n\geq 0\right)$ with respect to $\mathbb{P}\otimes\lambda$.

By Proposition \ref{supprop}, we get that 
\begin{align*}
&\|Y^n\|^2_{\mathcal{S}^2} + \left\|Z^n\right\|_{L^2(W) }^2 + \left\|U^n\right\|_{L^2(\tilde N) }^2  \leq e^{C_1(1+ C_{K,n})^2} \left(\E |\E_n\xi|^2+\E (I_{\E_nF})^2\right),
\end{align*}
where $C_{K,n}= \left\|\int_0^T\left(\E_n K_1(s)+(\E_n K_2(s))^2\right)ds\right\|_\infty$.
By the monotonicity of $\E_n$ and Jensen's inequality, we get that
\begin{align*}
 \int_0^T\!\!\left(\E_n K_1(s)+(\E_n K_2(s))^2\right)\!ds\leq\E_n\int_0^T\!\!\!(K_1(s)+K_2(s)^2)ds\leq C_K,\,\,\mathbb{P}\text{-a.s.}
\end{align*}
Doob's martingale inequality  applied to $n\mapsto\E_n\xi$ and 
$n\mapsto I_{\E_nF}=\E_n\int_0^TF(s)ds$ yields that
\begin{align*}
&\|Y^n\|^2_{\mathcal{S}^2} + \left\|Z^n\right\|_{L^2(W) }^2 + \left\|U^n\right\|_{L^2(\tilde N) }^2  \leq e^{C_1(1+ C_{K})^2} \left(\E |\xi|^2+\E (I_{F})^2\right).
\end{align*}

Furthermore, 
\begin{equation}\label{supns-conditional}
\sup_{n\geq 0} \left(\|\E_n Y\|_{\mathcal{S}^2} +\left\|\E_n Z\right\|_{L^2(W) }^2 + \left\|\E_n U\right\|_{L^2(\tilde N) }^2\right)<\infty
\end{equation}
follows from martingale convergence and Jensen's inequality and implies uniform integrability of the processes 
$\left({|\E_n Y|},n\geq 0\right)\text{ and }\left(|\E_n Z|+\|\E_n U\|,n\geq 0\right)$ with respect to $\mathbb{P}\otimes\lambda$.
\bigskip

\texttt{Step 3:}\\
In this step, we show the convergence \eqref{want}. 
From martingale convergence, we get that for all $t\in {[0,T]}$, $\E_n Y_t\to Y_t$, $\E_n Z_t\to Z_t$ and $\E_n U_t\to U_t$, $\mathbb{P}$-a.s.
This implies that $f(s,\E_n Y_s,\E_n Z_s,\E_n U_s)\to f(s,Y_s,Z_s,U_s)$ in $\mathbb{P}\otimes\lambda$. Therefore, 
\begin{align*}
\lim_{n\to\infty}&\E\int_0^T|Y^n_s-\E_n Y_s||f(s,\E_n Y_s,\E_n Z_s,\E_n U_s)- f(s,Y_s,Z_s,U_s)|\\
&\quad\quad\times\chi_{\{|Y^{n}_s|+|\E_n Y_s|\leq K\}}ds=0
\end{align*}
since the integrals form a uniformly integrable sequence with respect to $\mathbb{P}\otimes\lambda$. Indeed, we have,
using \ref{2} for $f$ and the first equation of \eqref{algebra}, the estimate
\begin{align*}
&|Y^n_s-  \E_n Y_s||f(s,\E_n Y_s,\E_n Z_s,\E_n U_s)- f(s,Y_s,Z_s,U_s)|\chi_{\{|Y^n_s|+|\E_n Y_s|\leq K\}} \\
&\leq  4K (F(s)+K_1(s))\nonumber\\
&\quad
+2K(K_2(s)^2)+|Z_s|^2+\|U_s\|^2+|\E_n Z_s|^2+\|\E_n U_s\|^2),\nonumber
\end{align*}
where  $ n\mapsto \E_n Z_s,  n\mapsto \E_n  U_s$ converge since they are 
closable martingales. 
 
Next, we will show that 
\begin{align}\label{want3}
\delta_n(K):= \E \bigg \{\int_0^T&|Y^n_s-\E_n Y_s||f(s,\E_n Y_s,\E_n Z_s,\E_n U_s)- f(s,Y_s,Z_s,U_s)|\nonumber\\
&\times\chi_{\{|Y^n_s|+|\E_n Y_s|> K\}}ds \bigg \}
\end{align}
can be made arbitrarily small by the choice of $K>0,$ uniformly in $n$. Again by \ref{2} and using the notation $\chi^n_K(s):= \chi_{\{|Y^n_s|+|\E_n Y_s|> K\}},$
we estimate like in \eqref{A2-used}  
\begin{align*}
&|Y^n_s-\E_n Y_s||f(s,\E_n Y_s,\E_n Z_s,\E_n U_s)- f(s,Y_s,Z_s,U_s)|\chi_{\{|Y^n_s|+|\E_n Y_s|> K\}} \\
&\le |Y^n_s-  \E_n Y_s| \Big(2F(s)+K_1(s)(|Y_s|+|\E_n Y_s|)\\
&\quad\quad\quad\quad\quad+K_2(s)(|Z_s|+|\E_n Z_s|+\|U_s\|+\|\E_n U_s\|)\Big)\chi^n_K(s)
\end{align*}
and get
\begin{align} \label{A2-used-again} 
\delta_n(K) &\le   2\E \bigg \{\int_0^T\chi^n_K(s)\,  F(s)ds  \,\,  (\sup_{r\in{[0,T]}}| Y^n_r|+\sup_{r\in{[0,T]}}|\E_n Y_r|) \bigg \}\nonumber \\
&\quad +   \E \bigg \{\int_0^T \chi^n_K(s)\,K_2(s)\nonumber\\
&\quad\quad\quad\quad\times(| Y^n_s|+|\E_n Y_s|)(|Z_s|+|\E_n Z_s|+\|U_s\|+\|\E_n U_s\|)ds \bigg \}  \nonumber \\
&\quad +\E\int_0^T   |Y^n_s-\E_n Y_s|\, (|Y_s|+|\E_n Y_s|)   \chi^n_K(s) \,K_1(s)ds \nonumber\\
&=: \delta^{(1)}_{n,K} + \delta^{(2)}_{n,K}+ \delta^{(3)}_{n,K}.
\end{align}
For $\delta^{(1)}_{n,K},$ we estimate
$$\delta^{(1)}_{n,K}\leq  2  \left\|\int_0^T\chi^n_K(s) F(s)ds\right\|_2\sup_{l\geq 0}(\|Y^l\|_{\mathcal{S}^2}+\|\E_l Y_s\|_{\mathcal{S}^2})$$
which tends to zero as $K\to\infty$, since  we have  $\chi^n_{K}\to 0$   in $\mathbb{P}\otimes\lambda$, uniformly in $n,$   as $K\to \infty.$ The latter is implied by  the uniform integrability of the families $(|Y^n|)_{n\geq 0}$ and $(|\E_n Y|)_{n\geq 0}$ with respect to $\mathbb{P}\otimes\lambda.$  
We continue with the next summands, 
\begin{align}\label{delta2} 
&\delta^{(2)}_{n,K}\le 8\bigg ( \E\int_0^T(| Y^n_s|^2 +|\E_n Y_s|^2)\chi^n_K(s)\,K_2(s)^2 
 ds\bigg)^\frac{1}{2} \notag  \\
& \hspace{8em}\times\left ( \left\|Z\right\|_{L^2(W)}+ \left\|U\right\|_{L^2(\tilde N) }\right) 
\end{align}
and 
\begin{align}\label{delta3}
&\delta^{(3)}_{n,K}\le  \E\int_0^T (|Y_s|^2+|Y^n_s|^2 + 2|\E_n Y_s|^2)\,  \chi^n_K(s) \,K_1(s)ds,
\end{align}
where, for  $\E\int_0^T\chi^n_K(s) ( |Y_s|^2 +|Y^n_s|^2)K_1(s)ds$ and $\E\int_0^T \chi^n_K(s)  |Y^n_s|^2  K_2(s)^2ds,$ we will apply the estimate \eqref{eta-half-trick} from the proof of Lemma \ref{lemma-eta-trick}. For example (the other terms can be treated similarly), we get   
\equal \label{eta-trick-term}
&&\E\int_0^T   \chi^n_K(s)    | Y^n_s|^2\,K_2(s)^2 ds \notag \\
&\leq & \,   \E\int_0^T\chi^n_K(s)\,K_2(s)^2ds\cdot  e^{\int_0^T \eta_n(s)ds}|\E_n\xi|^2\nonumber\\
&& + 2\E\int_0^T\int_0^s\chi^n_K(s)\, K_2(\tau)^2d\tau\, e^{\int_0^s \eta_n(\tau)d\tau}\E_n F(s)|Y^n_s|ds \nonumber\\
& \leq & e^{2 C_K}\E\int_0^T\chi^n_K(s)\,K_2(s)^2ds|\E_n \xi|^2\nonumber\\
&& +2e^{2 C_K} \left\|\int_0^T \chi^n_K(s)\,K_2(s)^2ds\cdot I_{\E_n F}\right\|_2\|Y^n\|_{\mathcal{S}^2}
\tionl
with $\int_0^T \eta_n(s)ds =\int_0^T \E_n K_1(s)+(\E_n K_2(s))^2ds \le C_K$  a.s. 
Now, 
one gets that
$$\int_0^T \chi^n_K(s) K_2(s)^2ds \stackrel{\PP}{\to} 0, \quad K\to \infty. $$
Furthermore, using $\sup_{n\geq 0}\E_n\int_0^TF(s)ds<\infty$, $\mathbb{P}$-a.s. (which follows from martingale convergence),
$$\int_0^T \chi^n_K(s) K_2(s)^2ds\int_0^T\E_nF(s)ds \stackrel{\PP}{\to} 0, \quad K\to \infty,$$ independently of $n$.
Since, by Doob's maximal inequality,
\begin{align*}
&\E \left[\sup_{n\geq 0}\int_0^TK_2(s)^2ds\int_0^T\E_nF(s)ds\right]^2\\
&\leq \E \left[\sup_{n\geq 0}C_K \E_n\int_0^T F(s) ds\right]^2\leq  4C_K^2 \E I_F^2<\infty,
\end{align*}
dominated convergence is applicable to the last expression in \eqref{eta-trick-term}. The first summand containing $\xi$ can be treated in the same way.

The terms containing $|\E_n Y_s|$ in the inequalities \eqref{delta2} and \eqref{delta3}, e.g.,~the expression $\E\int_0^T \chi^n_K(s) |\E_n Y_s|^2K_1(s)ds,$ can be estimated by 
\begin{align*}
\E\int_0^T \chi^n_K(s) |\E_n Y_s|^2K_1(s)ds&\leq \E\int_0^T \chi^n_K(s) K_1(s)\left(\sup_{l\geq 0}|\E_l Y_s|\right)^2ds\\
&\leq\E \bigg \{\int_0^T\chi^n_K(s) K_1(s)ds\left(\sup_{t\in{[0,T]}}\sup_{l\geq 0}\E_l|Y_t|\right)^2  \bigg \}\\
&\leq 2C_K\|Y\|^{ 2}_{\mathcal{S}^2},
\end{align*}
where we used Doob's maximal inequality again. Since $\int_0^T\chi^n_K(s) K_1(s)ds\to 0$ in $\mathbb{P}$ as $K\to \infty$, all the terms in \eqref{delta2} and \eqref{delta3} become small, uniformly in $n$, if $K$ is large.
So the expressions $\delta^{(2)}_{n,K}$ and $\delta^{(3)}_{n,K}$ can be made arbitrarily small by the choice of $K$, which gives us the  desired convergence
\begin{equation*}
\E\int_0^T |Y^n_s-\E_n Y_s||f(s,\E_n Y_s,\E_n Z_s,\E_n U_s)- f(s,Y_s,Z_s,U_s)|ds\to 0.
\end{equation*}

\texttt{Step 5:}

Since, by the last step, $$\|Y^n-\E_n Y\|_{L^2(W)}+\|Z^n-\E_n Z\|_{L^2(W)}+\|U^n-\E_n U\|_{L^2(\tilde N)}\to 0,$$
and also, by martingale convergence,
$$\|Y-\E_n Y\|_{L^2(W)}+\|Z-\E_n Z\|_{L^2(W)}+\|U-\E_n U\|_{L^2(\tilde N)}\to 0,$$
we get

$$\|Y^n- Y\|_{L^2(W)}+\|Z^n-Z\|_{L^2(W)}+\|U^n-U\|_{L^2(\tilde N)}\to 0.$$
\end{proof}

\section*{Acknowledgement}
The authors  thank  Stefan Geiss and Juha Ylinen, University of Jyv\"askyl\"a, for fruitful discussions and valuable suggestions.\smallskip

{\ch Moereover, we are sincerly grateful to the anonymous reviewers for their helpful comments and questions. }\smallskip

Christel Geiss would like to thank the  Erwin Schr\"odinger Institute, Vienna, for hospitality and support, where  a part of this work was written.

\subsection*{Funding}

Large parts of this article were written when Alexander Steinicke was member of the Institute of Mathematics and Scientific Computing, University of Graz, Austria, and supported by the Austrian Science Fund (FWF): Project F5508-N26, which is
part of the Special Research Program ``Quasi-Monte Carlo Methods: Theory and Applications.''

\appendix
\section{Appendix}
{\bf The Bihari--LaSalle inequality.} 
For the  Bihari--LaSalle inequality   we refer to \cite[pp. 45-46]{maobook}. Here, we formulate a  backward version of it which  has been applied  in \cite{YinMao}.
The proof is analogous to that in \cite{maobook}.

\begin{prop} \label{bihari-prop} Let $c>0.$ Assume that $\rho: [0,\infty[ \to [0,\infty[$ is a continuous and non-decreasing function such that $\rho(x) >0$ for all 
$x>0.$ Let $K$ be a non-negative, integrable Borel function on $[0,T],$ and $y$ a non-negative, bounded Borel function on $[0,T],$ such that
\begin{align*}
y(t) &\le c + \int_t^T K(s) \rho(y(s)) ds.
\end{align*}   
Then, it holds that 
$$ y(t) \le G^{-1} \left (G(c) +   \int_t^T K(s)  ds \right ) $$

for all $t \in [0,T]$ such that $G(c) +   \int_t^T K(s)  ds \in {\rm dom}(G^{-1}).$
Here 
\begin{align*}
 G(x) := \int_1^x \frac{dr}{\rho(r)},
\end{align*}
and $G^{-1}$ is the inverse function of $G.$ \\
Especially, if $\rho(r)=r$ for $r\in [0,\infty[,$ it holds that
\equal \label{gronwall-backwards}
y(t)\le c e^{\int_t^T K(s)  ds}.
\tionl
\end{prop}

{\bf The  Jankov--von Neumann theorem.}
If $X$ and $Y$ are sets and $P\subseteq  X \times Y,$  then  $P^*\subseteq P$ is called a {\it uniformization} of $P$ if and only if $P^*$ is the graph of a function  
$f : {\rm proj}_X(P) \to Y,$ i.e.,  $P^*= \{(x,f(x)):  x \in {\rm proj}_X(P)\}.$ 
Such a function $f$ is called  a {\it uniformizing function} for $P$. 
Let $ \Sigma_1^1(X)$ denote the class of analytic  subsets of $X.$
The following theorem can be found, for example, in \cite[Theorem 18.1]{Kechris}.

\begin{thm}[Jankov--von Neumann theorem]  \label{JvN} Assume that $X$ and $Y$ are standard Borel spaces and $P \subseteq  X \times Y$ is an analytic set.
Then, $P$ has a uniformizing function  that is  $\sigma(\Sigma_1^1(X))$- measurable.

\end{thm}

 \bibliographystyle{plain}

\end{document}